\documentclass[11pt,a4paper]{amsart}

\usepackage[english]{babel}
\usepackage[utf8]{inputenc}
\usepackage[T1]{fontenc}
\usepackage{amsmath,amssymb,amsthm,mathtools}
\usepackage{enumitem}
\usepackage{booktabs}
\usepackage{aliascnt}
\usepackage{hyperref}
\usepackage{cleveref}
\usepackage{graphicx}
\usepackage[margin=2.5cm]{geometry}
\usepackage{xcolor}
\usepackage{algorithm}
\usepackage{algpseudocode}
\usepackage{float}
\usepackage{caption}
\usepackage{subcaption}
\usepackage{multirow}
\usepackage{placeins}
\usepackage{tikz}
\usetikzlibrary{positioning,arrows.meta}

\newtheorem{theorem}{Theorem}[section]

\newaliascnt{proposition}{theorem}
\newtheorem{proposition}[proposition]{Proposition}
\aliascntresetthe{proposition}

\newaliascnt{lemma}{theorem}
\newtheorem{lemma}[lemma]{Lemma}
\aliascntresetthe{lemma}

\newaliascnt{corollary}{theorem}
\newtheorem{corollary}[corollary]{Corollary}
\aliascntresetthe{corollary}

\newaliascnt{conjecture}{theorem}
\newtheorem{conjecture}[conjecture]{Conjecture}
\aliascntresetthe{conjecture}

\theoremstyle{definition}
\newaliascnt{definition}{theorem}
\newtheorem{definition}[definition]{Definition}
\aliascntresetthe{definition}

\newaliascnt{remark}{theorem}
\newtheorem{remark}[remark]{Remark}
\aliascntresetthe{remark}

\newaliascnt{example}{theorem}

\aliascntresetthe{example}

\newaliascnt{observation}{theorem}
\newtheorem{observation}[observation]{Observation}
\aliascntresetthe{observation}

\newcommand{\R}{\mathbb{R}}

\newcommand{\one}{\mathbf{1}}
\newcommand{\Var}{\operatorname{Var}}

\newcommand{\spec}{\operatorname{spec}}

\newcommand{\Tr}{\operatorname{Tr}}

\title[Phase Transition and Spectral Structure in Finite Free Information Inequalities]{Spectral Structure in Finite Free Information Inequalities and $p$-Stam Phase Transitions}

\author{Baran Hashemi}
\address{Max Planck Institute for Mathematics in the Sciences, Leipzig, Germany}
\email{baran.hashemi@mis.mpg.de}

\begin{document}

\begin{abstract}
Using FlowBoost, a closed-loop deep generative optimization framework for extremal structure discovery, we investigate $\ell^p$-generalizations of the finite free Stam inequality for real-rooted polynomials under finite free additive convolution~$\boxplus_n$. At $p=2$, FlowBoost finds the Hermite pair as the unique equality case and reveals the spectral structure of the linearized convolution map at this extremal point. As a result, we conjecture that the singular values of the doubly stochastic coupling matrix~$E_n$ on the mean-zero subspace are ${2^{-k/2}:k=1,\ldots,n-1}$, independent of~$n$. Conditional on this conjecture, we obtain a sharp local stability constant and the finite free CLT convergence rate, both uniform in~$n$. We introduce a one-parameter family of $p$-Stam inequalities using $\ell^p$-Fisher information and prove that the Hermite pair itself violates the inequality for every $p>2$, with the sign of the deficit governed by the $\ell^p$-contraction ratio of~$E_n$. Systematic computation via FlowBoost supports the conjecture that $p^*\!=2$ is the sharp critical exponent. For $p<2$, the extremal configurations undergo a bifurcation, meaning that they become non-matching pairs with bimodal root structure, converging back to the Hermite diagonal only as $p\to 2^-$. Our findings demonstrate that FlowBoost, can be an effective tool of mathematical discovery in infinite-dimensional extremal problems.
\end{abstract}

\maketitle

\tableofcontents

\section{Introduction}\label{sec:intro}

This project was initiated by the \emph{First~Proof} benchmark~\cite{abouzaid2026proof}, a collection of ten open research-level problems contributed by research mathematicians as a challenge to AI systems. That work demonstrated the potential and the difficulty of AI-driven mathematical discovery, and motivated us to apply FlowBoost~\cite{FlowBoost25} and Abductive Hypothesis Testing (AHT)~\cite{hashemi2025can} to its fourth problem.

\subsection{Background}
Let $f(x) = \prod_{i=1}^n (x - \alpha_i)$ be a monic degree-$n$ polynomial with simple real roots
$\alpha_1 < \cdots < \alpha_n$. Following the definition in~\cite{GVS26} (using $s_n$ in place of $\mathcal{J}_n$ therein), we define the \emph{score vector}

\begin{equation}\label{eq:score}
  s_n(\alpha)_i
  \;:=\;
  \sum_{j \neq i} \frac{1}{\alpha_i - \alpha_j},
  \qquad 1 \le i \le n,
\end{equation}
and the \emph{finite free Fisher information}
\begin{equation}\label{eq:Phi}
  \Phi_n(f)
  \;:=\;
  \frac{4}{n(n-1)^2}
  \sum_{i=1}^n
  \bigg(\sum_{j \neq i} \frac{1}{\alpha_i - \alpha_j}\bigg)^{\!2}
  \;=\;
  \frac{4}{n(n-1)^2}\,\|s_n(\alpha)\|_2^2.
\end{equation}
Given two monic real-rooted polynomials $f, g$ of degree $n$ with roots $\alpha, \beta$, their \emph{finite free additive convolution} is
\begin{equation}\label{eq:boxplus}
  (f \boxplus_n g)(x)
  \;:=\;
  \frac{1}{n!}\sum_{\pi \in S_n}
  \prod_{i=1}^n (x - \alpha_i - \beta_{\pi(i)}),
\end{equation}
equivalently computed via the coefficient formula~\cite{AP18,MSS22}: if $f(x) = \sum_{k=0}^n (-1)^k a_k x^{n-k}$ and $g(x) = \sum_{k=0}^n (-1)^k b_k x^{n-k}$ with $a_0 = b_0 = 1$, then
\begin{equation}\label{eq:boxplus_coeff}
  c_k
  \;=\;
  \sum_{i+j=k}
  \frac{(n-i)!\,(n-j)!}{n!\,(n-k)!}\,a_i\,b_j,
  \qquad k = 0,\ldots,n.
\end{equation}
The product $f \boxplus_n g$ is real-rooted whenever $f$ and $g$ are so~\cite{Walsh22}. The \emph{finite free Stam inequality}, proved by Garza-Vargas, Srivastava, and Stier~\cite{GVS26}, states

\begin{theorem}[Finite free Stam inequality
  {\cite[Theorem~1.4]{GVS26}}]
\label{thm:stam}
  For all monic real-rooted $f, g$ of degree $n$,
  \begin{equation}\label{eq:stam}
    \frac{1}{\Phi_n(f \boxplus_n g)}
    \;\ge\;
    \frac{1}{\Phi_n(f)} + \frac{1}{\Phi_n(g)}.
  \end{equation}
  Equality holds when $f$ and $g$ are Hermite polynomials.
\end{theorem}

The proof relies on (i)~a Jacobian identity relating score vectors through the root map
$\Omega_{\boxplus_n}: (\alpha,\beta) \mapsto \gamma$, and (ii)~a positivity result for Hessians of root maps from the convexity theory of hyperbolic
polynomials~\cite{BGLS01}. Crucially, step~(ii) is specific to the $\ell^2$ structure of $\Phi_n$, namely the Hessian positivity gives
$\|w\|_2^2 - \|J_{\boxplus_n}(w)\|_2^2 \ge 0$, a statement about \emph{quadratic forms} that does not readily extend to $\ell^p$ norms for $p \neq 2$.

While~\cite{GVS26} establishes \Cref{thm:stam} and identifies the Hermite equality case, the proof does not address: (a)~whether Hermite is the \emph{unique} equality
case (up to affine transformations), (b)~any \emph{quantitative stability} estimate controlling the deficit by distance from the extremal set, or (c)~the
detailed \emph{spectral structure} of the Jacobian contraction at the Hermite point that governs both stability and CLT convergence rates. These questions motivate the computational investigation presented here.

\subsection{Contributions}

This paper presents four contributions, all discovered or evidenced through FlowBoost and AHT.

\begin{enumerate}[label=\textbf{C\arabic*.},leftmargin=2em]

\item \textbf{Independent discovery of the Hermite
  extremizer (\Cref{sec:hermite}).} Using FlowBoost to minimize the scale-free deficit $\rho$ at $p = 2$, we independently discover that equality configurations converge to the Hermite root configuration with $f \approx g$. This computational discovery predates our reading of the proof in~\cite{GVS26}.

\item \textbf{Geometric spectrum of the convolution coupling
  (\Cref{sec:spectrum}).}
  Guided by the discovered extremizer, we compute the doubly
  stochastic coupling matrix $E_n$ at the Hermite point and
  discover that its singular values on the mean-zero subspace
  numerically follow $\{2^{-k/2} : k = 1, \ldots, n-1\}$, independent of $n$, supported by a high-precision restricted audit through $n \le 90$. Conditional on this conjecture, one obtains the sharp local stability constant ($1/2$, uniform in $n$) and the finite free CLT convergence rate ($1/\!\sqrt{2}$, uniform in $n$).

\item \textbf{Sharp phase transition for the $p$-Stam inequality (\Cref{sec:pstam}).} We introduce the one-parameter family of \emph{$p$-Stam inequalities} with $\ell^p$-Fisher information and discover a sharp phase transition at $p^* = 2$ where the inequality holds for $p \in (1, 2]$ and fails for $p > 2$, with counterexamples over all tested $(n, p)$. Then, we prove that the Hermite pair itself violates the inequality for $p > 2$ (\Cref{prop:hermite_counter}), connecting the phase transition mechanism to the spectrum of~$E_n$.

\item \textbf{Extremizer bifurcation for $p < 2$
  (\Cref{sec:extremizers_p}).} Combining FlowBoost's global search with AHT, for $p \lesssim 1.8$, the extremal points become \emph{non-matching} pairs with \emph{bimodal} root structure, lying far from Hermite. No local analysis starting at the Hermite point can detect these configurations. As $p \to 2^-$, the bimodal structure merges and the extremizers converge to the Hermite pair (\Cref{conj:extremizer_transition}). The e-value-based AHT screen strongly rejects the Hermite family in favor of a two-block structure throughout the subcritical regime, while the precise closed-form description of the subcritical extremizer family remains open.

\end{enumerate}

\begin{remark}[Relation to the classical setting]
    The classical Stam inequality goes back to Blachman~\cite{Blachman65}. For R\'enyi entropy power, Bobkov and Chistyakov proved a constant-factor entropy power inequality for all orders $\alpha>1$~\cite{BobkovChistyakov15}, while Bobkov and Marsiglietti proved additive exponentiated variants of the form $N_r^\beta(X+Y)\ge N_r^\beta(X)+N_r^\beta(Y)$ for $\beta \ge (r+1)/2$~\cite{BobkovMarsiglietti17}. In free probability, the corresponding information inequalities originate in Voiculescu's free Fisher information~\cite{Voiculescu93}, the free entropy power inequality of Szarek and Voiculescu~\cite{SzarekVoiculescu96}, and Shlyakhtenko's free Shannon analogue~\cite{Shlyakhtenko07}. Our \Cref{conj:phase} asserts that the finite free analogue has the same critical exponent $p^* = 2$, suggesting that this threshold reflects a deep structural boundary.
\end{remark}

Sections~\ref{sec:hermite}–\ref{sec:extremizers_p} present the four contributions in the order they were discovered, preceded by setup (\Cref{sec:setup}) and a description of the FlowBoost and AHT pipelines (\Cref{sec:flowboost}).

\section{Setup and Definitions}\label{sec:setup}

\subsection{The \texorpdfstring{$p$}{p}-Fisher information and the \texorpdfstring{$p$}{p}-Stam inequality}

For $p > 1$ and a monic real-rooted polynomial $f$ of degree $n$ with roots $\alpha$, define the \emph{$p$-Fisher information}
\begin{equation}\label{eq:Phi_p}
  \Phi_{n,p}(f)
  \;:=\;
  \|s_n(\alpha)\|_p^p
  \;=\;
  \sum_{i=1}^n
  \bigg|\sum_{j \neq i}
  \frac{1}{\alpha_i - \alpha_j}\bigg|^p.
\end{equation}
For $p = 2$ this reduces to $\|s_n(\alpha)\|_2^2 = \frac{n(n-1)^2}{4}\,\Phi_n(f)$.
Since $\Phi_n(f)$ and $\Phi_{n,2}(f)$ differ only by the positive constant
$\frac{4}{n(n-1)^2}$, the $p$-Stam deficit $g_p$ and its sign are unaffected by
which normalization is used. All theorems and tables in this paper involving
$g_p$ use the unnormalized $\Phi_{n,p}$ of~\eqref{eq:Phi_p}.

The \emph{$p$-Stam inequality} asserts, for all monic real-rooted $f, g$ of degree $n$:
\begin{equation}\label{eq:pstam}
  \frac{1}{\Phi_{n,p}(f \boxplus_n g)^{1/(p-1)}}
  \;\ge\;
  \frac{1}{\Phi_{n,p}(f)^{1/(p-1)}}
  +
  \frac{1}{\Phi_{n,p}(g)^{1/(p-1)}}.
\end{equation}

The exponent $1/(p-1)$ is dictated by affine dilation. Under $f(x) \mapsto f_c(x) := c^n f(x/c)$ and $g(x) \mapsto g_c(x) := c^n g(x/c)$, the roots scale by $(\alpha,\beta,\gamma) \mapsto (c\alpha,c\beta,c\gamma)$,
the score vector scales by $s \mapsto c^{-1}s$, and hence
\[
  \Phi_{n,p}(f_c)=c^{-p}\Phi_{n,p}(f),\qquad
  \Phi_{n,p}(g_c)=c^{-p}\Phi_{n,p}(g).
\]
Therefore
$\Phi_{n,p}^{-1/(p-1)}$ is homogeneous of degree $p/(p-1)$, and $1/(p-1)$ is an exponent for which
both sides of~\eqref{eq:pstam} scale with the same degree under a common dilation of the pair. For $p = 2$, the
exponent is $1$ and~\eqref{eq:pstam} reduces to~\eqref{eq:stam}.

\begin{definition}[$p$-Stam deficit]\label{def:gp}
  For monic real-rooted $f, g$ of degree $n$ and $p > 1$, define the \emph{additive deficit}
  \begin{equation}\label{eq:gp}
    g_p(f,g)
    \;:=\;
    \frac{1}{\Phi_{n,p}(f \boxplus_n g)^{1/(p-1)}}
    -
    \frac{1}{\Phi_{n,p}(f)^{1/(p-1)}}
    -
    \frac{1}{\Phi_{n,p}(g)^{1/(p-1)}}.
  \end{equation}
  The $p$-Stam inequality~\eqref{eq:pstam} holds if and only if $g_p(f,g) \ge 0$ for all $(f, g)$.
\end{definition}

\begin{definition}[Scale-free deficit]\label{def:rho_p}
  The \emph{baseline denominator} and
  \emph{scale-free $p$-deficit} are
  \begin{equation}\label{eq:rho_p}
    A_p(f,g) :=
    \frac{1}{\Phi_{n,p}(f)^{1/(p-1)}}
    + \frac{1}{\Phi_{n,p}(g)^{1/(p-1)}},
    \qquad
    \rho_p(f,g) := \frac{g_p(f,g)}{A_p(f,g)}.
  \end{equation}
  The $p$-Stam inequality holds if and only if
  $\rho_p(f,g) \ge 0$ for all $(f,g)$.
  For $p = 2$, we write $\rho := \rho_2$.
\end{definition}

\subsection{Distance diagnostics and normalization}

Let $h^{(n)} \in \R^n$ denote the sorted roots of the degree-$n$ probabilists' Hermite polynomial $\mathrm{He}_n$, normalized to have mean $0$ and variance $1$. To compare root \emph{shapes} across different samples, we normalize each root vector $r \in \R^n$ by
\begin{equation}\label{eq:normalize}
  r
  \;\mapsto\;
  \widetilde r
  :=
  \frac{r - \bar r \mathbf{1}}
       {\sqrt{\frac{1}{n}\sum_{i=1}^n (r_i-\bar r)^2}},
  \qquad
  \bar r := \frac{1}{n}\sum_{i=1}^n r_i .
\end{equation}
This normalization is used only for geometric diagnostics and visualization. It defines a shape-based comparison between root configurations, but it is not a symmetry reduction of the underlying variational problem, that is the finite free Stam objective is invariant under a \emph{common} affine transformation applied to both polynomials, not under independent rescalings.

For a normalized root vector $r \in \R^n$, we define the Hermite distance
\begin{equation}\label{eq:dH}
  d_H(r)
  :=
  \frac{1}{\sqrt n}
  \min\Bigl\{
    \|r-h^{(n)}\|_2,\,
    \|r-\mathrm{rev}(h^{(n)})\|_2
  \Bigr\},
\end{equation}
where $\mathrm{rev}$ denotes reversal. For a pair $(\alpha,\beta)$ of normalized root vectors, we also define
\begin{equation}\label{eq:D_dPQ}
  D(\alpha,\beta) := d_H(\alpha)^2 + d_H(\beta)^2,
  \qquad
  d_{PQ}(\alpha,\beta) := \frac{1}{\sqrt n}\|\alpha-\beta\|_2.
\end{equation}

\subsection{Computational Pipeline: FlowBoost}
\label{sec:flowboost}
FlowBoost~\cite{FlowBoost25} is a closed-loop generative optimization framework for constrained continuous search in high-dimensional spaces. In FlowBoost, one trains a stochastic policy to \emph{sample} from a distribution proportional to a reward function $R(x)$, while trying to find its maximum. The search over the state space is shaped by a balance condition between the probability mass entering and leaving each region, so that the trained generator explores the full reward landscape and naturally covers multiple high-reward modes. The ``flow'' is a learned ODE that transports a tractable prior (Gaussian) to a terminal distribution concentrated on high-reward configurations, trained via flow matching and reward-guided fine-tuning so that the generator progressively focuses on the extremal region.

In the present problem, a state is a pair of degree-$n$ root configurations
\[
x = (\alpha, \beta) \in \R^n \times \R^n,
\]
and the reward is $R(x) = -g_p(\alpha,\beta)$ (we seek the pair that makes the deficit as small, or as negative, as possible).
Feasibility consists of simple real roots, a prescribed minimum separation, and the normalization convention used in the corresponding experiment.
A running top-$K$ buffer retains the highest-reward feasible configurations found so far; this buffer plays the role of the replay memory in a generative flow network and provides the training signal for successive refinement rounds.
The generator does not collapse to a single maximiser because the flow is trained to cover the reward landscape proportionally, so that diverse extremal configurations are retained alongside the best ones.

Each outer loop has four stages:

\begin{enumerate}[nosep,label=\arabic*.]
  \item \textbf{Stochastic random perturbation (SRP):}
    starting from prior samples, run a projected
    derivative-free local search with several random
    restarts. Each restart alternates Gaussian
    perturbations, symmetry moves, and
    projection back to the feasible set, namely it is followed by
    a short random-direction refinement with step-size
    reduction. In later closed-loop rounds, the same
    procedure is initialized from the saved buffer of best
    samples produced in the previous round.
  \item \textbf{Supervised Conditional Flow Matching:}
    use the top fraction of valid SRP samples,
    and train a Transformer-based conditional flow-matching~\cite{Vaswani17,lipman2022flow}
    model to transport a prior distribution toward
    this extremal region.
  \item \textbf{Reward-guided fine-tuning:}
    guide the model toward lower-deficit configurations by fine-tuning it with the oracle score (evaluated deficit) as reward and the self-distillation consistency regularization term.
  \item \textbf{Generative sampling and final push:}
    draw new candidates by ODE integration of the learned
    flow from the prior, then apply the same projected
    local search as a final refinement step. The best valid
    configurations are retained in the running top-$K$
    buffer.
\end{enumerate}

For a candidate pair $(\alpha, \beta)$, the oracle first
computes the monic coefficient vectors of the corresponding
polynomials, forms the finite free convolution coefficients
via~\eqref{eq:boxplus_coeff}, and extracts the roots
$\gamma$ of $f \boxplus_n g$. The objective is then evaluated from the corresponding Fisher-information functional. For the $p$-Stam, we use the additive deficit $g_p(f,g) = F_p(f \boxplus_n g) - F_p(f) - F_p(g)$ with $F_p(f) = \Phi_{n,p}(f)^{-1/(p-1)}$. We use the normalization
\[
  \widetilde{\Phi}_{n,p}(f)
  :=
  \frac1n\sum_{i=1}^n \left|\frac{2}{n-1}s_i(f)\right|^p
  =
  \frac{(2/(n-1))^p}{n}\,\Phi_{n,p}(f),
\]
so
$\widetilde{F}_p = c_{n,p}^{-1/(p-1)} F_p$ and
$\widetilde{g}_p = c_{n,p}^{-1/(p-1)} g_p$
for a positive constant $c_{n,p}$ depending only on
$(n,p)$.
The sign of the deficit is therefore the same for $g_p$ and
$\widetilde{g}_p$, and all numerical values reported in tables
and figures throughout this paper are given in terms of the
unnormalized deficit $g_p$ defined in~\eqref{eq:gp}.
Our dedicated $p=2$ recovery experiments minimize $\rho$
for $n \in \{6, 8, 10, 12, 20, 50\}$.
The $p$-Stam sweep uses the same closed-loop pipeline for
$n \in \{6, 8, 10, 12, 20\}$ and $p \in [1.05, 6.0],$
with objective $g_p$. For each tested pair $(n,p)$ with
$p<2$, we perform the structural post-analysis described in
\Cref{sec:extremizers_p}.

\subsection{Computational Pipeline: Abductive Hypothesis Testing (AHT)}
\label{sec:families}

To interpret and characterize the extremal roots discovered by FlowBoost, we use the Abductive Hypothesis Testing (AHT) framework of~\cite{hashemi2025can}. The defining principle of AHT is \emph{inference to the best explanation} where rather than deducing consequences from known axioms, one infers the simplest structural hypothesis consistent with observations produced by an AI system. In its original formulation, AHT is part of a broader interpretability machinery, combining abductive reasoning, representation probing, and causal inference, used to extract human-readable mathematical conjectures from trained neural networks. 

In our setting we adapt AHT to a generative optimization loop, with FlowBoost (\Cref{sec:flowboost}) as the discovery engine, and we augment it as an \emph{E-value-based AHT} in the betting sense of Shafer~\cite{Shafer21}. For each tested pair $(n,p)$, let $\mathcal{E}_{n,p} = \{x^{(1)},\ldots,x^{(M_{n,p})}\}$
denote the retained elite set of extremizers from FlowBoost, where each $x^{(m)}=(\alpha^{(m)},\beta^{(m)})$ is a pair of root configurations. We then compute a summary map $T(\mathcal{E}_{n,p})$ consisting of low-dimensional structural statistics extracted from the elite population. For each candidate family $\mathcal{H}$ in a fixed hypothesis library, AHT uses these summaries to formulate the working explanatory null hypothesis
\[H_0(\mathcal{H};n,p):
  \text{``Best $\mathrm{Top}_k$ Samples~}\mathcal{E}_{n,p}
  \text{ is explained by }\mathcal{H}\text{''},
  \]
against the alternative that the family-fit residuals and associated descriptors are too large or too unstable across the best population. Because the data are optimizer-selected elites rather than i.i.d. observations, classical $p$-values and asymptotic significance levels are not appropriate. We therefore use a conditional e-value~\cite{VovkWang21}, together with the population-level fit summaries, to decide whether a candidate explanation survives comparison with the Hermite baseline and remains stable across neighboring values of $(n,p)$. The AHT loop thus has four stages:

\begin{enumerate}[nosep,label=(\roman*)]
  \item \textbf{Generate.} FlowBoost returns elite extremal sets for each
    tested $(n,p)$.
  \item \textbf{Describe.} The summary statistics $T(\mathcal{E}_{n,p})$ is extracted from each elite population, including family-fit residuals, pair mismatch, Hermite distance, and symmetry frequencies.
  \item \textbf{Screen by e-value.} Candidate model classes $\mathcal{H}$ are compared against the Hermite baseline through one-sided conditional e-values and through the stability of their fit summaries across neighboring $(n,p)$ values.
  \item \textbf{Reject or promote.} A candidate family is promoted to a structural conjecture only if this combined fit-and-evidence screen remains favorable across the tested grid, otherwise it is rejected in favor of the alternative.
\end{enumerate}

The e-value screen is defined as follows. For fixed $(n,p)$ and a candidate family $\mathcal{H}$, let $R_{\mathcal{H}}(x^{(m)}) = d^{\mathcal{H}}_{\mathrm{joint}}(x^{(m)})$ be the Mode~A joint residual and define the win indicator
\[
  Z_m^{\mathcal{H}}
  :=
  \mathbf{1}\!\left\{
    R_{\mathcal{H}}(x^{(m)}) < R_{\mathrm{He}}(x^{(m)})
  \right\},
\]
after discarding ties. Under the \emph{working} model that
the retained point cloud is an exchangeable sample from a fixed
law, we test
\[
  H_{0}^{\mathrm{He}}(\mathcal{H};n,p):
  \mathbb{P}(Z_m^{\mathcal{H}}=1)\le \tfrac12
\]
against the one-sided alternative
$\mathbb{P}(Z_m^{\mathcal{H}}=1)>\tfrac12$ using the
Bernoulli-mixture e-value, in the same mixture
likelihood-ratio spirit that underlies time-uniform
Bernoulli testing and confidence-sequence constructions
\cite{HowardRamdasMcAuliffeSekhon21},
\begin{equation}\label{eq:evalue_screen}
  E_{\mathcal{H}}
  =
  \int_{1/2}^{1}
  \prod_{m=1}^{M_{\mathrm{eff}}}
  \frac{q^{Z_m^{\mathcal{H}}}(1-q)^{1-Z_m^{\mathcal{H}}}}
       {(1/2)^{Z_m^{\mathcal{H}}}(1/2)^{1-Z_m^{\mathcal{H}}}}
  \,\pi(dq),
  \qquad
  \pi = \mathrm{Unif}[1/2,1],
\end{equation}
where $M_{\mathrm{eff}}$ is the number of non-tied
comparisons. Under this working model, $E_{\mathcal{H}}\ge 20$ and $E_{\mathcal{H}}\ge 100$
correspond to nominal levels $0.05$ and $0.01$ by Markov's inequality;
because the elite point cloud is optimizer-selected rather than exchangeable,
these thresholds are diagnostic benchmarks rather than calibrated significance levels. The numerical results are reported in \Cref{app:evalues}.

\begin{figure}[htb!]
\centering
\begin{tikzpicture}[
    >=Stealth,
    node distance=0.8cm and 0.8cm,
    ahtstep/.style={
      rectangle,
      rounded corners=3pt,
      draw=black!65,
      fill=black!4,
      very thick,
      minimum width=2.9cm,
      minimum height=1.0cm,
      align=center,
      font=\small
    },
    ahtnote/.style={
      rectangle,
      rounded corners=3pt,
      draw=black!20,
      fill=black!2,
      minimum width=2.9cm,
      minimum height=0.82cm,
      align=center,
      font=\scriptsize
    },
    ahtflow/.style={->, thick, draw=black!75}
]

\node[ahtstep] (gen) {FlowBoost Sampling};
\node[ahtstep, right=of gen] (desc) {Describe};
\node[ahtstep, right=of desc] (hyp) {Screen by $E$};
\node[ahtstep, right=of hyp] (sel) {Reject / Promote};

\node[ahtnote, below=0.55cm of gen] (gennote) {Top-k elite populations\\$\mathcal{E}_{n,p}$};
\node[ahtnote, below=0.55cm of desc] (descnote) {Summary statistics\\$T(\mathcal{E}_{n,p})$};
\node[ahtnote, below=0.55cm of hyp] (hypnote) {Null hypothesis and e-values\\$H_0^{\mathrm{He}}(\mathcal{H};n,p),\,E_{\mathcal H}$};
\node[ahtnote, below=0.55cm of sel] (selnote) {Reject $H_0$ or promote\\an alternative};

\draw[ahtflow] (gen) -- (desc);
\draw[ahtflow] (desc) -- (hyp);
\draw[ahtflow] (hyp) -- (sel);

\end{tikzpicture}
\caption{E-value-based AHT as used in this study. FlowBoost first produces elite extremal samples $\mathcal{E}_{n,p}$. These are compressed into summary statistics $T(\mathcal{E}_{n,p})$, compared against a library of low-complexity candidate families $\mathcal{H}$, and screened against the Hermite baseline through the conditional e-values of \eqref{eq:evalue_screen}. A candidate family is promoted only when its population-level fit remains stable across the tested grid and its e-value evidence is persistently favorable.}
\label{fig:aht_pipeline}
\end{figure}
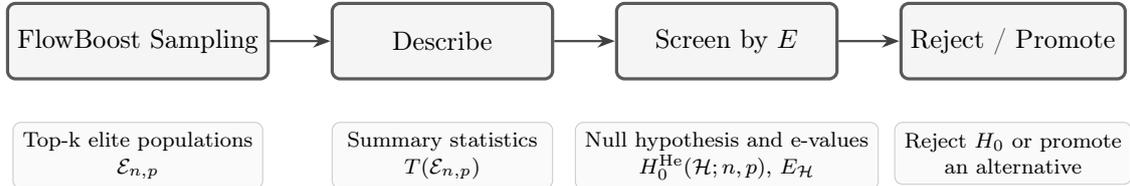

This abductive workflow is summarized in \Cref{fig:aht_pipeline}. The library of candidate structural hypotheses tested at stage~(iii) consists of: Hermite ($\mathrm{He}_n$), semicircle quantiles, uniform spacing, Jacobi families ($J_n^{(a,b)}$ for several $(a,b)$), and \emph{two-block uniform} configurations, where the roots split into two equally spaced clusters of sizes $\lfloor n/2 \rfloor$ and $\lceil n/2 \rceil$ separated by a gap. Fitting is performed in the per-polynomial gauge via affine residual minimization:
\begin{equation}\label{eq:dfamily}
  d_{\mathrm{family}}(r)
  \;:=\;
  \min_{a > 0}
  \frac{1}{\sqrt{n}}
  \|r - a\,r_{\mathrm{ref}}\|_2,
\end{equation}
and the \emph{joint residual} for a pair is
$d_{\mathrm{joint}} = \sqrt{d_{\mathrm{family}}(\alpha)^2
+ d_{\mathrm{family}}(\beta)^2}$.

\section{The \texorpdfstring{$p = 2$}{p = 2} Case: Hermite Extremizer Recovery}
\label{sec:hermite}

Before the proof in~\cite{GVS26} became available,
we applied FlowBoost at $p = 2$ with the sole objective of
minimizing the scale-free deficit~$\rho$, for $n \in \{6,
8, 10, 12, 20, 50\}$. FlowBoost samples converged to the Hermite root configuration with $f \approx g$, with no prior knowledge that this was the equality case.

\begin{table}[htb!]
\centering
\caption{FlowBoost results at $p = 2$.
$\rho_{\min}$ is the best scale-free deficit;
$D_{\min}$ and $d_{PQ,\min}$ are the Hermite distance
and mismatch at the best sample.}
\label{tab:p2_summary}
\begin{tabular}{@{}r l l l l@{}}
\toprule
$n$ & $\rho_{\min}$ & $D_{\min}$
    & $d_{PQ,\min}$ & Spearman$(\rho, D)$ \\
\midrule
6  & $2.1 \times 10^{-9}$  & $5.3 \times 10^{-10}$
   & $1.2 \times 10^{-5}$  & 0.89 \\
8  & $2.2 \times 10^{-8}$  & $3.2 \times 10^{-9}$
   & $6.6 \times 10^{-5}$  & 0.87 \\
10 & $3.6 \times 10^{-7}$  & $4.0 \times 10^{-8}$
   & $1.0 \times 10^{-5}$  & 0.85 \\
12 & $8.2 \times 10^{-7}$  & $1.1 \times 10^{-8}$
   & $4.6 \times 10^{-4}$  & 0.80 \\
20 & $2.0 \times 10^{-7}$  & $2.4 \times 10^{-8}$
   & $3.6 \times 10^{-4}$  & 0.80 \\
50 & $8.4 \times 10^{-6}$  & $5.0 \times 10^{-7}$
   & $1.0 \times 10^{-4}$  & 0.79 \\
\bottomrule
\end{tabular}
\end{table}

As $\rho \to 0$, both $D \to 0$ and $d_{PQ} \to 0$ simultaneously, meaning that the equality forces both root vectors toward the Hermite configuration \emph{and} toward each other, with no exception across all tested degrees.

\begin{table}[htb!]
\centering
\caption{Best sample at $n = 6$: root vectors
$\alpha, \beta$ vs.\ normalized Hermite roots
$h^{(6)}$, displayed to four decimal places.
The apparent discrepancy at $i=6$ ($1.4867$ vs.\ $1.4866$) is
a rounding artifact; the true residual $D_{\min}=5.3\times10^{-10}$
(see \Cref{tab:p2_summary}) is consistent with agreement to nine significant figures.}
\label{tab:top_samples}
\begin{tabular}{@{}c r@{\;\;}r@{\;\;}r@{\;\;}r@{\;\;}r@{\;\;}r@{}}
\toprule
& $i = 1$ & $i = 2$ & $i = 3$ & $i = 4$ & $i = 5$ & $i = 6$ \\
\midrule
$h^{(6)}_i$ & $-1.4866$ & $-0.8449$ & $-0.2758$
            & $\phantom{-}0.2758$ & $\phantom{-}0.8449$ & $\phantom{-}1.4866$ \\
$\alpha_i$  & $-1.4866$ & $-0.8449$ & $-0.2758$
            & $\phantom{-}0.2758$ & $\phantom{-}0.8449$ & $\phantom{-}1.4867$ \\
$\beta_i$   & $-1.4866$ & $-0.8449$ & $-0.2758$
            & $\phantom{-}0.2758$ & $\phantom{-}0.8449$ & $\phantom{-}1.4867$ \\
\bottomrule
\end{tabular}
\end{table}

\begin{observation}[Hermite extremizer at $p = 2$]
\label{obs:hermite}
   In the unnormalized setting, the pair
  $(f,g)=(\mathrm{He}_n,\mathrm{He}_n)$ is an equality case
  of the finite free Stam inequality, and more generally
  every matching affine image of this pair also satisfies
  $g_2(f,g)=0$~\cite[Theorem~1.4]{GVS26}. In the
  per-polynomial gauge~\eqref{eq:normalize}, this yields
  $\rho(\alpha,\beta)=0$ whenever
  $\alpha=\beta=h^{(n)}$ up to reversal symmetry.
\end{observation}

\begin{remark}\label{rem:uniqueness}
  Our search strongly support the converse
  uniqueness statement, namely the
  Hermite diagonal appears to be the only equality case.
  The proof in~\cite{GVS26} proceeds via the identity
  $\|w\|^2 - \|J_{\boxplus_n}(w)\|^2
  = \sum_i \Omega_{\boxplus_n,i}\, H_{\boxplus_n,i}(w, w) \ge 0$,
  where the inequality comes from G\aa{}rding/Bauschke
  convexity of hyperbolic polynomials. We expect that a
  suitable strict-convexity refinement away from the
  Hermite orbit would convert this numerical uniqueness
  into a theorem, but we do not pursue that here.
\end{remark}

\section{Spectral Structure of the Convolution Coupling}
\label{sec:spectrum}
Having identified the Hermite pair as the extremizer of the deficit at $p=2$, it does not yet describe the local geometry of the inequality near equality. The next question is therefore: how does the deficit respond to small perturbations of the root vectors around the Hermite configuration, and how does the convolution root map propagate such perturbations? Both questions are governed by the linearization of the map at the Hermite diagonal. We therefore study the Jacobian of the convolution root map at that point.

\subsection{The coupling matrix \texorpdfstring{$E_n$}{En}}
Let $h = h^{(n)} \in \R^n$ be the normalized Hermite root
vector. The root map $\Omega_{\boxplus_n}: (\alpha,\beta) \mapsto \gamma$ is smooth at any point where $\alpha$ has simple roots~\cite[Observation~2.11]{GVS26}. Define
\begin{equation}\label{eq:En}
  E_n
  \;:=\;
  \frac{\partial\,\Omega_{\boxplus_n}(\cdot,\, h)}
       {\partial\alpha}\bigg|_{\alpha = h}
  \;\in\; \R^{n \times n}.
\end{equation}
By Lemma~3.4 of~\cite{GVS26}, $E_n$ is a doubly stochastic
matrix with nonnegative entries. At the symmetric point $(\alpha,\beta)=(h,h)$, the symmetry $\Omega_{\boxplus_n}(f,g)=\Omega_{\boxplus_n}(g,f)$ implies, by differentiating both sides at the diagonal
\begin{equation}\label{eq:sym_jacobian}
  E_n
  \;=\;
  \frac{\partial\,\Omega_{\boxplus_n}(h,\,\cdot)}
       {\partial\beta}\bigg|_{\beta = h},
\end{equation}
so, at the Hermite diagonal, the full Jacobian satisfies $J_{\boxplus_n}[u\oplus v]=E_n(u+v)$. Thus the linearization depends only on the sum of the two perturbations. In particular, anti-diagonal perturbations $u = -v$ lie in the kernel of $J_{\boxplus_n}|_V$. Let
\[
\one=(1,\dots,1)^\top,
\qquad
W:=\one^\perp=\{u\in\R^n:\ u^\top\one=0\}.
\]
Since $E_n\one=\one$ and $\one^\top E_n=\one^\top$, the subspace $W$ is invariant under $E_n$. Accordingly, $\R^n=\mathrm{span}(\one)\oplus W$ splits into the translation direction $\mathrm{span}(\one)$ and the mean-zero subspace $W$, where the nontrivial local geometry is encoded. Restricting $E_n$ to the mean-zero subspace $W=\one^\perp$, we audit that induced operator directly. A dedicated high-precision computation based on
implicit differentiation at the Hermite diagonal computes
the first ten singular values of $E_n|_W$ for
$n\in\{10,20,30,40,50,60,70,80,90\}$ using multiprecision
coefficient arithmetic in \texttt{mpmath}~\cite{mpmath};
details are deferred to \Cref{app:spectrum_audit}. The outcome is that across the full audited range, the leading singular values closely follows the dyadic targets $2^{-k/2}$, independent of $n$. That is why we propose the following conjecture,

\begin{conjecture}[Geometric spectrum of $E_n$]
\label{conj:spectrum}
  Let $n \ge 3$ and let $E_n$ be defined
  by~\eqref{eq:En}. The singular values of $E_n$ on
  $W = \one^\perp$ are
  \begin{equation}\label{eq:spectrum}
    \spec(E_n|_W)
    \;=\;
    \big\{2^{-k/2} : k = 1, \ldots, n-1\big\},
  \end{equation}
  independent of $n$, with singular value $1$ on
  $\mathrm{Span}(\one)$.
\end{conjecture}

\begin{remark}[Symmetry of $E_n$]
\label{rem:En_sym}
Numerically, $E_n$ appears to be symmetric, meaning its eigenvalues
coincide with its singular values.  The argument via
equation~\eqref{eq:sym_jacobian} establishes only that
$\partial_\alpha \Omega = \partial_\beta \Omega$ at the diagonal,
which implies that both first-order partial Jacobians agree, but does
not by itself imply that $E_n$ is a symmetric matrix.  We therefore
conjecture symmetry as a separate numerical observation: across all
audited degrees, $\|E_n - E_n^\top\|_F / \|E_n\|_F < 10^{-10}$.
\end{remark}

This means that the $k$-th singular value is $2^{-k/2}$ regardless of the degree of the polynomial. \Cref{tab:spectrum} reports selected low-mode values. The full analysis indicates a degree-independent low-mode profile with new smaller modes appended only at the tail. If \Cref{conj:spectrum} holds and $E_n$ is moreover symmetric (as conjectured in \Cref{rem:En_sym}), then its eigenvalues coincide with its singular values. In particular,
\[
  \Tr(E_n|_W)=\sum_{k=1}^{n-1}2^{-k/2}
  \longrightarrow \frac{2^{-1/2}}{1-2^{-1/2}}
  = \sqrt{2}+1
  \qquad (n\to\infty),
\]
so the trace converges to the finite limit $\sqrt2+1$ as $n\to\infty$.

\begin{table}[htb!]
\centering
\caption{Singular values of $E_n|_W$ vs.\ the conjectured
  $\sigma_k = 2^{-k/2}$, together with
  the maximum relative error among the first ten modes.}
\label{tab:spectrum}
\begin{tabular}{@{}r l l l l l l@{}}
\toprule
$n$ & $\sigma_1$ & $\sigma_2$ & $\sigma_3$
    & $\sigma_4$ & $\sigma_5$
    & $\mathrm{rel.err.}$ \\
\midrule
\textit{conj.}
    & $0.707106781$ & $0.500000000$ & $0.353553391$
    & $0.250000000$ & $0.176776695$ & --- \\
\midrule
10 & 0.707106781 & 0.500000000 & 0.353553391 & 0.250000000 & 0.176776695 & $1.33\times10^{-14}$ \\
40 & 0.707106781 & 0.500000000 & 0.353553391 & 0.250000000 & 0.176776695 & $8.02\times10^{-11}$ \\
60 & 0.707106778 & 0.500000000 & 0.353553391 & 0.249999999 & 0.176776697 & $1.15\times10^{-7}$ \\
80 & 0.707104585 & 0.499997829 & 0.353554003 & 0.249999577 & 0.176779011 & $5.77\times10^{-5}$ \\
90 & 0.707145636 & 0.500091161 & 0.353518684 & 0.250002448 & 0.176746177 & $4.56\times10^{-4}$ \\
\bottomrule
\end{tabular}
\end{table}

\begin{remark}[Analogy with the derivative case]
\label{rem:derivative_analogy} 

Proposition~4.6 of~\cite{GVS26} computes a single spectral
quantity for the differentiation root map, namely the second
singular value of $J_{\partial_x}$ equals $\sqrt{(n-2)/n}$,
which depends on the degree. \Cref{conj:spectrum} is stronger on both counts as it specifies the \emph{entire} spectrum of the convolution root map, and the values are \emph{universal constants} $2^{-k/2}$ independent of~$n$.

\end{remark}

\begin{figure}[htb!]
\centering
\includegraphics[width=0.98\textwidth]{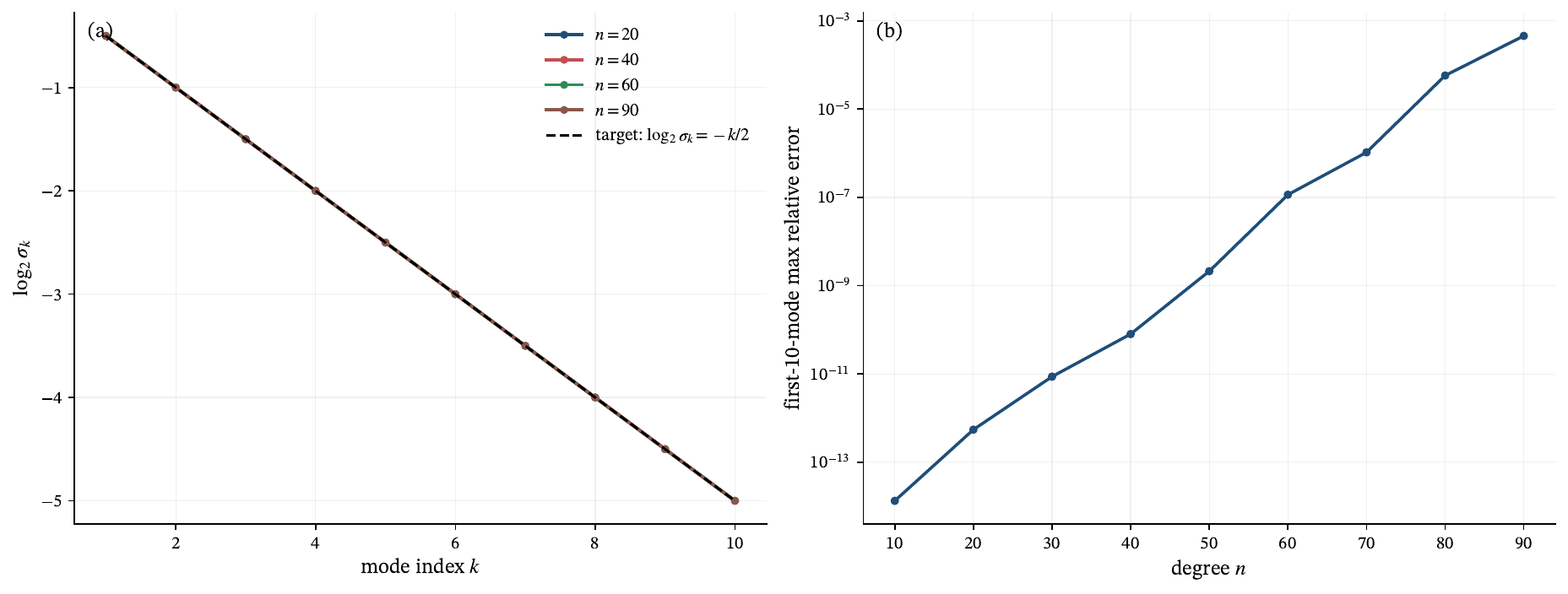}
\caption{High-precision numerical evidence for the
conjectured dyadic spectrum of the Hermite coupling
operator $E_n|_W$. Panel~(a) plots
$\log_2 \sigma_k$ for the first ten singular values at the
representative degrees $n=20,40,60,90$, together with the
reference line $\log_2 \sigma_k=-k/2$. Panel~(b) shows the
maximum relative error among the first ten singular values
across the high-precision restricted audit
$n\in\{10,20,30,40,50,60,70,80,90\}$.}
\label{fig:spectrum}
\end{figure}

\Cref{fig:spectrum} is the main numerical motivation for
\Cref{conj:spectrum}. Panel~(a) shows that the observed
dyadic law is not confined to a single degree, while
Panel~(b) quantifies that the first ten modes remain
accurate through the largest high-precision degree we
audited, namely $n=90$, with maximum relative error below
$5\times10^{-4}$. The numerical picture therefore provides a low-mode universality rather than degree-dependent drift.

\begin{corollary}[Full singular value decomposition of $J_{\boxplus_n}|_V$]
\label{cor:svd}
Assume \Cref{conj:spectrum}. Let $\{e_k\}_{k=1}^{n-1}$ be the eigenvectors of $E_n|_W$
with eigenvalues $\lambda_k = 2^{-k/2}$. On the subspace
$V = \{u \oplus v : u^\top\one = v^\top\one = 0\}
\subset \R^{2n}$:
the direction $(e_k, e_k)/\sqrt{2}$ maps to
$\sqrt{2} \cdot 2^{-k/2}\,e_k$, giving singular value
$2^{(1-k)/2}$; the direction $(e_k, -e_k)/\sqrt{2}$ maps
to~$0$. Thus
\begin{equation}\label{eq:svd_full}
  \sigma_k\big(J_{\boxplus_n}\big|_V\big)
  = 2^{(1-k)/2},
  \qquad k = 1, \ldots, n-1.
\end{equation}
In particular, $\sigma_1 = 1$
(recovering~\cite[Remark~5.4]{GVS26}), $\sigma_2 = 1/\!\sqrt{2}$,
and all subsequent singular values decay geometrically with ratio $1/\!\sqrt{2}$.
\end{corollary}

\begin{corollary}[Local stability constant]
\label{cor:stability}
  Assume \Cref{conj:spectrum}. Near the Hermite point,
  the quadratic defect of the Jacobian satisfies
  \begin{equation}\label{eq:contraction}
    \|w\|^2 - \|J_{\boxplus_n}(w)\|^2
    \;=\;
    \sum_{k=1}^{n-1}(1 - 2^{1-k})\,
    \Big\langle w, \frac{(e_k, e_k)}{\sqrt{2}}\Big\rangle^2
    \;+\;
    \sum_{k=1}^{n-1}
    \Big\langle w, \frac{(e_k, -e_k)}{\sqrt{2}}\Big\rangle^2.
  \end{equation}
  Here the anti-diagonal modes lie in the kernel of
  $J_{\boxplus_n}|_V$ and therefore contribute with
  coefficient $1$. The diagonal coefficients are
  $0,\frac12,\frac34,\frac78,\ldots$ for
  $k=1,2,3,4,\ldots$.
  Consequently, on the orthogonal complement of the neutral
  mode $(e_1,e_1)/\sqrt2$ one has the uniform bound
  \[
    \|w\|^2 - \|J_{\boxplus_n}(w)\|^2
    \;\ge\;
    \frac12\,\Big\|P_{\perp}w\Big\|^2,
  \]
  where $P_\perp$ denotes orthogonal projection onto
  $\mathrm{Span}\{(e_1,e_1)/\sqrt2\}^\perp$.
\end{corollary}

\begin{remark}
The smallest nonzero coefficient in the quadratic form is attained on the second diagonal mode $(e_2,e_2)/\sqrt{2}$ where every anti-diagonal mode is penalized more strongly, with coefficient $1$. This bottleneck mode is also responsible for the CLT convergence rate $\lambda_{\mathrm{CLT}} = 1/\sqrt{2}$ in \Cref{cor:clt} where the same eigenvector $e_2$ governs both the sharpest contraction bound and the slowest decay toward Hermite.
\end{remark}

\begin{theorem}[Finite free CLT~{\cite{Marcus21}}]
\label{thm:ff_clt}
The \emph{finite free central limit theorem} states that
iterated variance-normalized self-convolution converges to
the Hermite polynomial.
Define $f_0 := f$ and
\begin{equation}\label{eq:clt_iter}
  f_{k+1}
  \;:=\;
  \tfrac{1}{\sqrt{2}}{}_*(f_k \boxplus_n f_k),
\end{equation}
where $c_* g$ denotes the rescaling $c^n\,g(x/c)$.
Then $f_k \to \mathrm{He}_n$ (appropriately normalized)
as $k \to \infty$ for any real-rooted $f$ with
$\Var(f) = n - 1$.
\end{theorem}

\begin{corollary}[Linearized CLT rate]
\label{cor:clt}
  Assume \Cref{conj:spectrum}.
  At the Hermite fixed point, the linearized map for
  the iteration~\eqref{eq:clt_iter} on perturbations
  $\delta \in W$ acts as multiplication by the eigenvalues
  $\sqrt{2} \cdot \sigma_k(E_n|_W) = 2^{(1-k)/2}$.
  Let $W_2 := \mathrm{Span}\{e_2,\dots,e_{n-1}\} = W \cap \mathrm{Span}\{e_1\}^\perp$.
  The spectral radius of the linearized map on $W_2$ is
  \begin{equation}\label{eq:clt_rate}
    \lambda_{\mathrm{CLT}}
    \;=\;
    \sqrt{2}\,\sigma_2(E_n|_W)
    \;=\; \sqrt{2} \cdot \frac{1}{2}
    \;=\; \frac{1}{\sqrt{2}}
    \;\approx\; 0.7071,
  \end{equation}
  independent of $n$.
  Perturbations from the Hermite fixed point that lie in $W_2$
  decay geometrically at rate $(1/\!\sqrt{2})^k$ per iteration.
  Perturbations in the neutral direction $e_1$ (the leading eigenvector)
  are neither amplified nor contracted.
\end{corollary}

\begin{remark}
   The rate $\lambda_{\mathrm{CLT}} = 1/\sqrt{2}$ is analogous to the classical spectral picture where under the linearized CLT map, excess cumulants of order $j$ decay as $2^{(1-j)/2}$ per step, giving the same leading rate $1/\sqrt{2}$ for the slowest non-trivia mode~\cite{ABBN04,Barron86}. To our knowledge, the analogous linearized rate in the finite free setting has not been established previously (conditionally on \Cref{conj:spectrum}).
\end{remark}

\begin{figure}[htb!]
\centering
\includegraphics[width=0.72\textwidth]{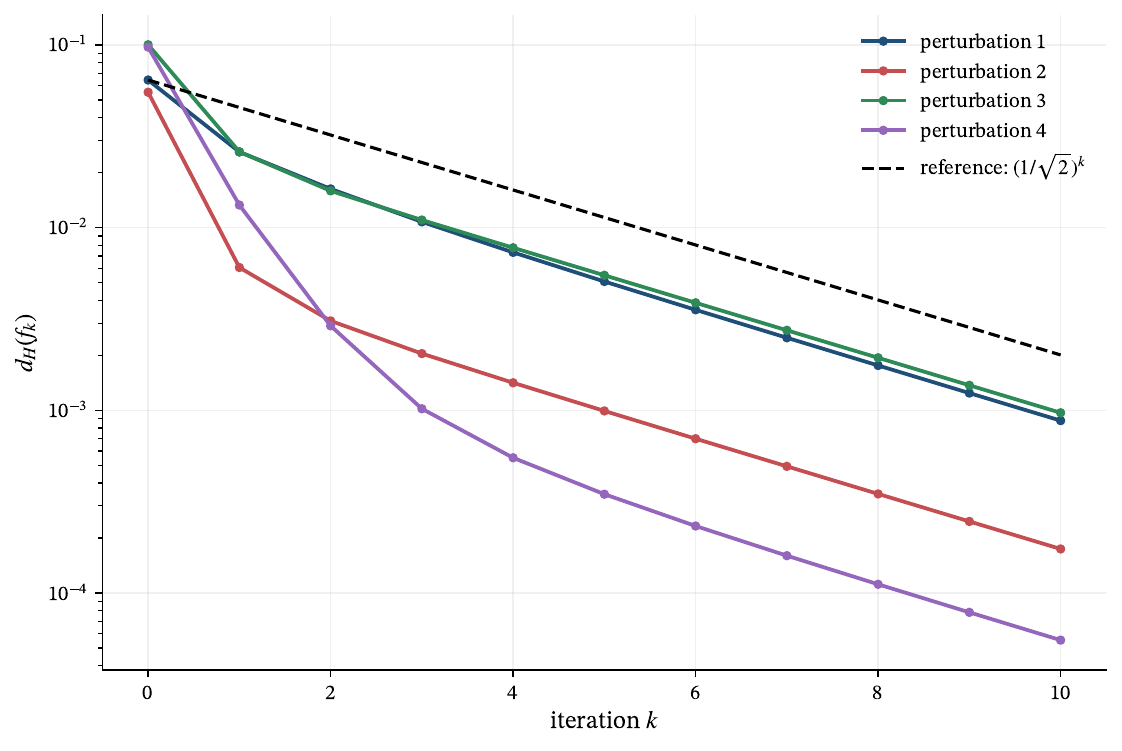}
\caption{Representative convergence trajectories for the
variance-normalized finite free central limit iteration at
$n=20$. Each colored curve tracks the perturbation norm
$d_H(f_k)$ for a different Hermite initialization,
while the dashed line gives the reference decay
$(1/\sqrt{2})^k$ predicted by the conjectural second
singular value. After a short transient, the observed
slopes are consistent with geometric contraction at the
predicted rate.}
\label{fig:clt}
\end{figure}

\Cref{fig:clt} shows how the conjectural spectrum governs the actual self-convolution dynamics, namely after a short transient, trajectories contract at the predicted rate $1/\sqrt{2}$ per doubling step.

\section{Phase Transition for the \texorpdfstring{$p$-Stam}{p-Stam} Inequality}
\label{sec:pstam}

The finite free Stam inequality is intrinsically tied to the
$\ell^2$ norm and the proof of~\cite{GVS26} rests on a
quadratic-form positivity that has no direct analogue for
$\ell^p$ with $p\neq 2$. This raises an immediate question, that is \emph{for which values of $p$ does an analogous $p$-Stam inequality hold?} The answer, supported by FlowBoost's exhaustive
counterexample search and an analytic result at the Hermite
point, is that $p=2$ is a sharp threshold. Based on our FlowBoost and AHT experimental results, we propose the following

\begin{conjecture}[$p$-Stam phase transition]
\label{conj:phase}
  For all $n \ge 3$, the $p$-Stam inequality~\eqref{eq:pstam} holds for all monic real-rooted pairs if and only if $p \in (1, 2]$. That is, the critical exponent is $p^* = 2$, independent of the degree $n$.
\end{conjecture}

\begin{conjecture}[$p$-Stam validity for $p \le 2$]
\label{conj:pstam_holds}
  For all $n \ge 3$ and $p \in (1, 2]$,
  the $p$-Stam inequality~\eqref{eq:pstam} holds for all
  monic real-rooted pairs $(f, g)$ of degree $n$.
\end{conjecture}

\Cref{tab:pstam_gt2} records the most negative deficits found on the tested grid for $p>2$. In this regime the point is not merely that counterexamples exist, the search stage already produces negative samples in every tested cell. Rather, the computation shows that the sign change at $p=2$ is robust, that the optimized negative margins vary only mildly with $n$, and that the strongest counterexamples concentrate near the Hermite diagonal (\Cref{obs:counter_hermite}).

\begin{proposition}[Robust negative $p$-Stam deficits for $p>2$]
\label{prop:pstam_fails}
For each $n \in \{6,8,10,12,20\}$ and each tested $p>2$, the computational pipeline produces explicit polynomial pairs $(\alpha^*,\beta^*)$ with $g_p(\alpha^*,\beta^*)<0$. These negative values remain negative under both standard and high-precision \texttt{mpmath} re-evaluation. On the tested grid, the optimized minima lie between approximately $-0.022$ and $-0.446$ and vary only mildly with $n$ at fixed $p$.
\end{proposition}

\Cref{tab:pstam_gt2} reports these optimized minima. Every entry is negative.

\begin{table}[htb!]
\centering
\caption{Minimum $p$-Stam deficit $g_p$ for $p > 2$.
All entries are negative.}
\label{tab:pstam_gt2}
\begin{tabular}{@{}r r r r r r@{}}
\toprule
$p$ & $n = 6$ & $n = 8$ & $n = 10$ & $n = 12$ & $n = 20$ \\
\midrule
2.0625 & $-0.0400$ & $-0.0400$ & $-0.0394$ & $-0.0379$
       & $-0.0222$ \\
2.125  & $-0.0744$ & $-0.0743$ & $-0.0738$ & $-0.0719$
       & $-0.0585$ \\
2.25   & $-0.1303$ & $-0.1302$ & $-0.1296$ & $-0.1275$
       & $-0.1114$ \\
2.5    & $-0.2086$ & $-0.2084$ & $-0.2082$ & $-0.2051$
       & $-0.1838$ \\
3.0    & $-0.2987$ & $-0.2983$ & $-0.2972$ & $-0.2949$
       & $-0.2775$ \\
4.0    & $-0.3819$ & $-0.3814$ & $-0.3811$ & $-0.3787$
       & $-0.3627$ \\
6.0    & $-0.4455$ & $-0.4457$ & $-0.4437$ & $-0.4448$
       & $-0.4259$ \\
\bottomrule
\end{tabular}
\end{table}

\begin{figure}[htb!]
\centering
\includegraphics[width=0.96\textwidth]{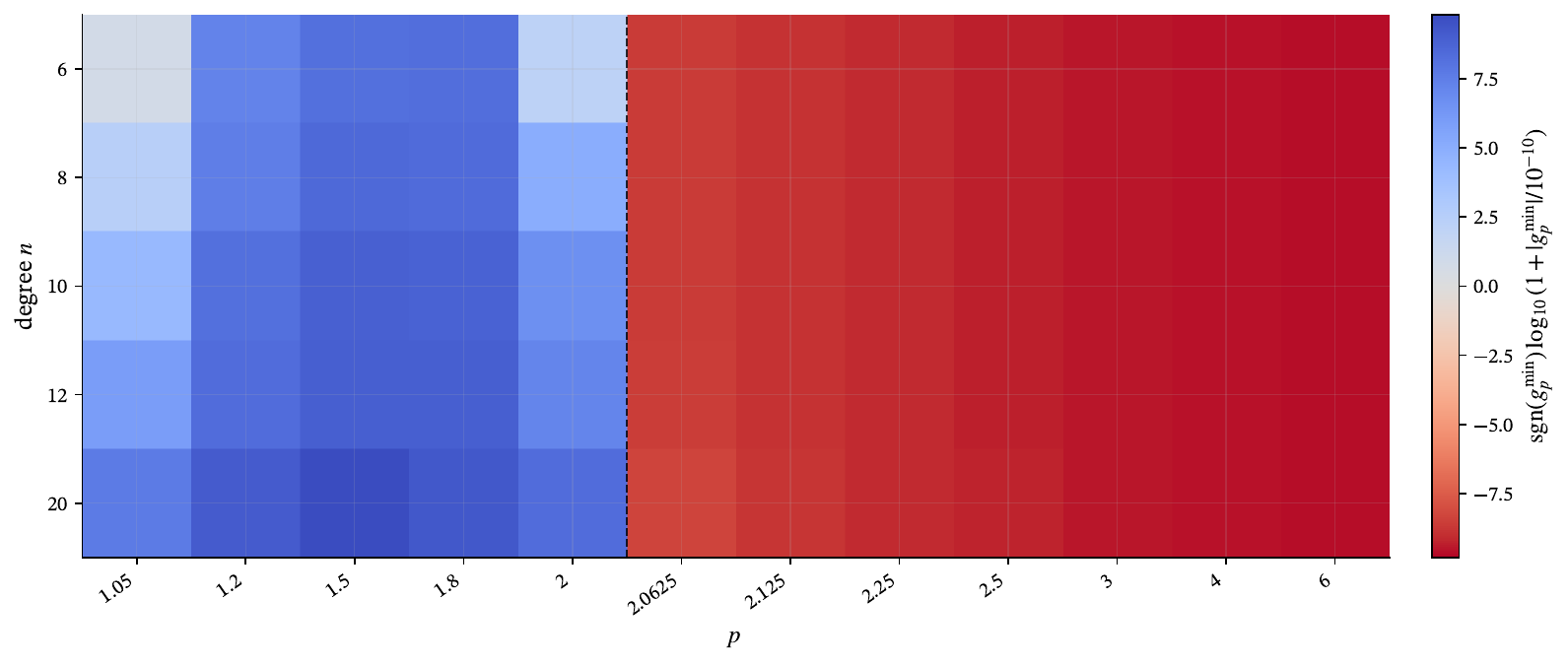}
\caption{Global computational phase diagram for the
$p$-Stam problem on the tested $(n,p)$ grid. The color in
each cell is $\operatorname{sgn}(g_p^{\min})\log_{10}(1+|g_p^{\min}|)$,
so the sign records whether the best observed deficit is positive or negative and the magnitude records its logarithmic size. The sign transition occurs at $p=2$ on the
tested grid, while the post-processed numerical reporting in
the text and tables is given in terms of the additive deficit $g_p$. Visually, the figure separates the observed validity
region $p \le 2$ from the counterexample region $p > 2$
in a single degree-uniform picture.}
\label{fig:phase_heatmap}
\end{figure}

\Cref{fig:phase_heatmap} summarizes the computational evidence on the tested grid.

\begin{observation}[Counterexamples concentrate at Hermite]
\label{obs:counter_hermite}
  For all $p > 2$, the counterexamples minimizing $g_p$
  have Hermite distance $D \sim 10^{-5}$--$10^{-3}$ and
  mismatch $d_{PQ} \sim 10^{-4}$--$10^{-2}$ (\Cref{tab:pstam_gt2_structure}). That is, the worst-case configurations for $p > 2$ are Hermite matching pairs.
\end{observation}

\begin{table}[htb!]
\centering
\caption{Structure of the best counterexamples for $p > 2$.}
\label{tab:pstam_gt2_structure}
\begin{tabular}{@{}r r c c@{}}
\toprule
$n$ & $p$ & $D(\alpha^*,\beta^*)$ & $d_{PQ}(\alpha^*,\beta^*)$ \\
\midrule
6 & 2.0625 & $1.4 \times 10^{-5}$ & $1.6 \times 10^{-4}$ \\
8 & 2.0625    & $1.1 \times 10^{-5}$ & $9.3 \times 10^{-5}$ \\
10 & 2.0625 & $2.2 \times 10^{-5}$ & $1.5 \times 10^{-3}$ \\
12 & 2.0625 & $1.5 \times 10^{-4}$ & $1.2 \times 10^{-2}$ \\
20 & 2.0625 & $9.3 \times 10^{-4}$ & $2.4 \times 10^{-2}$ \\
\bottomrule
\end{tabular}
\end{table}

\begin{proposition}[Hermite deficit for $p \neq 2$]
\label{prop:hermite_counter}
  Assume $n \ge 2$.
  Let $h^{(n)} \in \R^n$ denote the ordered roots of
  $\mathrm{He}_n$, normalized to mean $0$ and variance $1$,
  and let $s_n(h^{(n)}) \in \R^n$ be its score vector.
  Let $E_n \in \R^{n \times n}$ be the Hermite coupling
  matrix defined in~\eqref{eq:En}, so that at the
  symmetric point $(\alpha,\beta) = (h^{(n)}, h^{(n)})$
  one has
  \[
    s_n(\gamma) = E_n\,s_n(h^{(n)}),
    \qquad
    \gamma
    \;:=\;
    \Omega_{\boxplus_n}\!\big(h^{(n)}, h^{(n)}\big).
  \]
  Define the $\ell^p$-contraction ratio
  \begin{equation}\label{eq:contraction_ratio}
    \eta_p
    \;:=\;
    \frac{\|E_n\,s_n(h^{(n)})\|_p}
         {\|s_n(h^{(n)})\|_p}.
  \end{equation}
  Then the $p$-Stam deficit at the Hermite pair is
  \begin{equation}\label{eq:gp_hermite}
    g_p\!\big(h^{(n)}, h^{(n)}\big)
    \;=\;
    \|s_n(h^{(n)})\|_p^{-p/(p-1)}
    \Big[(\eta_p^p)^{-1/(p-1)} - 2\Big].
  \end{equation}
  In particular:
  \begin{enumerate}[nosep, label=(\alph*)]
    \item
      $g_p\!\big(h^{(n)}, h^{(n)}\big) = 0
      \;\Longleftrightarrow\;
      \eta_p^p = 2^{-(p-1)}
      \;\Longleftrightarrow\;
      \eta_p = 2^{-1 + 1/p}$.
    \item
      $g_p\!\big(h^{(n)}, h^{(n)}\big) < 0
      \;\Longleftrightarrow\;
      \eta_p^p > 2^{-(p-1)}$.
    \item
      $g_p\!\big(h^{(n)}, h^{(n)}\big) > 0
      \;\Longleftrightarrow\;
      \eta_p^p < 2^{-(p-1)}$.
  \end{enumerate}
  For $p = 2$, the finite free Stam inequality gives
  $g_2\!\big(h^{(n)}, h^{(n)}\big) = 0$, hence
  \[
    \eta_2^2 = 2^{-1}
    \qquad\text{and}\qquad
    \eta_2 = 2^{-1/2}.
  \]
\end{proposition}

\begin{proof}
  Set $h := h^{(n)}$, $s := s_n(h)$, and
  $\gamma := \Omega_{\boxplus_n}(h, h)$.
  By the one sided score transport identity at the symmetric point
  ($a=1,b=0$ in~\cite[Remark~3.6]{GVS26}), one has
  \[
    s_n(\gamma)
    \;=\;
    E_n\,s_n(h)
    \;=\;
    E_n\,s.
  \]
  Therefore
  \[
    \Phi_{n,p}(h \boxplus_n h)
    \;=\;
    \|s_n(\gamma)\|_p^p
    \;=\;
    \|E_n\,s\|_p^p.
  \]
  On the other hand,
  $\Phi_{n,p}(h) = \|s\|_p^p$.
  By definition of the $p$-Stam deficit,
  \[
    g_p(h, h)
    \;=\;
    \frac{1}{\Phi_{n,p}(h \boxplus_n h)^{1/(p-1)}}
    \;-\;
    \frac{1}{\Phi_{n,p}(h)^{1/(p-1)}}
    \;-\;
    \frac{1}{\Phi_{n,p}(h)^{1/(p-1)}}.
  \]
  Substituting gives
  \[
    g_p(h, h)
    \;=\;
    \frac{1}{\|E_n\,s\|_p^{p/(p-1)}}
    \;-\;
    \frac{2}{\|s\|_p^{p/(p-1)}}.
  \]
  Now write $A := \|s\|_p^p > 0$ and
  $\|E_n\,s\|_p^p = \eta_p^p\,A$.
  Then
  \[
    g_p(h, h)
    \;=\;
    (\eta_p^p\,A)^{-1/(p-1)} - 2\,A^{-1/(p-1)}
    \;=\;
    A^{-1/(p-1)}
    \Big[(\eta_p^p)^{-1/(p-1)} - 2\Big],
  \]
  which proves~\eqref{eq:gp_hermite}.
  Since $A^{-1/(p-1)} > 0$, the sign of $g_p(h, h)$ is
  the sign of $(\eta_p^p)^{-1/(p-1)} - 2$.
  Because $p > 1$, the map $x \mapsto x^{-1/(p-1)}$ is
  strictly decreasing on $(0, \infty)$.
  Hence
  \[
    g_p(h, h) = 0
    \;\;\Longleftrightarrow\;\;
    (\eta_p^p)^{-1/(p-1)} = 2
    \;\;\Longleftrightarrow\;\;
    \eta_p^p = 2^{-(p-1)}.
  \]
  This proves part~(a). Parts~(b) and~(c) follow immediately from the same strict monotonicity. Finally, for $p = 2$, the finite free Stam inequality yields $g_2(h, h) \ge 0$, and \cite[Theorem~1.4]{GVS26} shows that the Hermite pair is an equality case, so $g_2(h, h) = 0$. Applying part~(a) with $p = 2$ gives $\eta_2^2 = 2^{-1}$, as claimed.
\end{proof}

\begin{lemma}[Score eigenvector]
\label{lem:score_eigen}
  Let $h = h^{(n)}$ and $s = s_n(h)$.  Then
  $s$ is an eigenvector of $E_n$ with eigenvalue $2^{-1/2}$:
  \[
    E_n\,s \;=\; 2^{-1/2}\,s.
  \]
\end{lemma}
\begin{proof}
  By the CLT fixed-point property~\cite{Marcus21},
  $\Omega_{\boxplus_n}(h,h) = \sqrt{2}\,h$.
  The score transport identity (cf.~\cite[Remark~3.6]{GVS26})
  at the symmetric diagonal gives
  $s_n\!\bigl(\Omega_{\boxplus_n}(h,h)\bigr) = 2\,E_n\,s$.
  By the $(-1)$-homogeneity of the score vector under scaling,
  $s_n\!\bigl(\Omega_{\boxplus_n}(h,h)\bigr) = E_n\,s$.
  Combining gives $E_n\,s = 2^{-1/2}\,s$.
\end{proof}

The sign of the deficit in \Cref{prop:hermite_counter} is
now determined by \Cref{lem:score_eigen}.  Since
$E_n s = 2^{-1/2}s$, the contraction ratio defined
in~\eqref{eq:contraction_ratio} satisfies
$\eta_p = 2^{-1/2}$ independently of $p$ and $n$.
The critical threshold separating positive from negative
deficit is $\eta^* = 2^{-1 + 1/p}$, which is a strictly
decreasing function of $p$ that crosses $2^{-1/2}$
exactly at $p = 2$.  For $p < 2$ the threshold lies above
$\eta_p$, giving $g_p > 0$; for $p > 2$ it drops below,
giving $g_p < 0$.  \Cref{conj:spectrum} is not required for
this argument and it enters only when studying the action of
$E_n$ on perturbations beyond the score direction $s_n(h)$.

\subsection{Evidence for the $p \le 2$ validity}

For \texorpdfstring{$p \in \{1.05, 1.2, 1.5, 1.8, 2.0\}$}{p in {1.05, 1.2, 1.5, 1.8, 2.0}}, FlowBoost
consistently finds $g_p > 0$ across all tested degrees
(\Cref{tab:pstam_lt2}).

\begin{table}[htb!]
\centering
\caption{Minimum $p$-Stam deficit $g_p$ for the tested
subcritical grid and the quadratic endpoint.
All entries are positive (inequality supported).}
\label{tab:pstam_lt2}
\begin{tabular}{@{}r r r r r r@{}}
\toprule
$p$ & $n = 6$ & $n = 8$ & $n = 10$ & $n = 12$ & $n = 20$ \\
\midrule
1.05 & $4.47 \times 10^{-10}$ & $2.40 \times 10^{-8}$
     & $1.66 \times 10^{-6}$ & $8.22 \times 10^{-5}$
     & $3.90 \times 10^{-3}$ \\
1.2 & $1.81 \times 10^{-3}$ & $3.09 \times 10^{-3}$
    & $1.43 \times 10^{-2}$ & $2.21 \times 10^{-2}$
    & $1.17 \times 10^{-1}$ \\
1.5 & $1.46 \times 10^{-2}$ & $2.92 \times 10^{-2}$
    & $6.70 \times 10^{-2}$ & $8.50 \times 10^{-2}$
    & $6.48 \times 10^{-1}$ \\
1.8 & $1.69 \times 10^{-2}$ & $2.56 \times 10^{-2}$
    & $5.62 \times 10^{-2}$ & $8.66 \times 10^{-2}$
    & $1.81 \times 10^{-1}$ \\
2.0 & $1.30 \times 10^{-8}$ & $1.09 \times 10^{-5}$
    & $3.94 \times 10^{-4}$ & $1.38 \times 10^{-3}$
    & $1.94 \times 10^{-2}$ \\
\bottomrule
\end{tabular}
\end{table}

At $p = 2.0$, the minimum deficit approaches zero as FlowBoost converges toward the Hermite equality case, consistent with \Cref{thm:stam}. For $p < 2$, all observed minima remain positive.

\section{Extremizer Bifurcation for $p < 2$}
\label{sec:extremizers_p}
For $p<2$, the extremal configurations are not small perturbations of the Hermite pair detectable by local analysis around $p=2$. They are qualitatively different objects, namely non-matching polynomial pairs $(f \not\approx g)$ whose root distributions are bimodal, with roots clustering into two separated groups. In the strongly subcritical regime these configurations remain well away from the Hermite diagonal, so they are inaccessible to local analysis at $p=2$ and must be located by simulation-based optimization over the full $\R^{2n}$ landscape.
For instance, at $p=1.05$ the optimizer produces near-zero deficits at low degree while maintaining the same two-block geometry.

We apply AHT (\Cref{sec:families}) to interpret these configurations. FlowBoost elites are scored by invariant descriptors, stable patterns across the $(n,p)$ grid are promoted to conjectures, and inconsistent patterns are rejected. The two-block below are abductive hypothesis infered from the full extremal dataset, not outputs of any single sample.

For each $(n, p)$ with $p \in \{1.05, 1.2, 1.5, 1.8\}$,
the retained best configurations are analyzed along two
axes. Each root vector is matched against the family
library of~\Cref{sec:families} via~\eqref{eq:dfamily}
under two normalization modes: \emph{Mode~A} normalizes
each polynomial independently to mean~$0$, variance~$1$
before matching, while \emph{Mode~B} applies a common
affine transformation to the pair jointly and reports
the joint residual~$d_{\mathrm{joint}}$. Alongside
family fitting, we classify reflection and pairing
symmetries at thresholds $t \in \{0.05, 0.1, 0.2\}$,
recording S1~(per-polynomial palindromicity),
S2~(per-polynomial centrosymmetry),
S3~($\alpha \approx \beta$, diagonal matching), and
S4~($\alpha \approx -\mathrm{rev}(\beta)$,
anti-diagonal matching). Tracking the best-fitting
family, family residual, mismatch~$d_{PQ}$, Hermite
distance~$D$, and symmetry fractions as $p$ decreases
from~$2$ gives a continuous picture of the extremizer
bifurcation; the results are collected in
\Cref{tab:extremizer_full}. Three qualitatively distinct
regimes emerge.

\begin{table}[htb!]
\centering
\caption{Full extremizer structure as a function of $(n, p)$.
``2BU'' = two-block uniform, ``He'' = Hermite,
Consistency: fraction of the best set for which the
stated family achieves the lowest $d_{\mathrm{joint}}$.}
\label{tab:extremizer_full}
\small
\begin{tabular}{@{}r r l c c c c c@{}}
\toprule
$n$ & $p$ & Best & Med.\ $d_{\mathrm{jt}}$
 & Consist.\ & Best $d_{PQ}$ & Best $D$
    & Sym.\ (S3, $t\!=\!0.1$) \\
\midrule
6  & 1.05 & 2BU & 0.178 & 1.000 & 0.507 & 0.275 & 0\% \\
6  & 1.2 & 2BU & 0.444 & 1.000 & 0.496 & 0.263 & 0\% \\
6  & 1.5 & 2BU & 0.470 & 1.000 & 0.705 & 0.530 & 0\% \\
6  & 1.8 & He  & 0.034 & 0.281 & 0.020 & $4\!\times\!10^{-4}$ & 94\% \\
\midrule
8  & 1.05 & 2BU & 0.617 & 0.750 & 0.761 & 0.619 & 0\% \\
8  & 1.2 & 2BU & 0.603 & 0.926 & 0.842 & 0.647 & 0\% \\
8  & 1.5 & 2BU & 0.490 & 0.953 & 0.561 & 0.314 & 8\% \\
8  & 1.8 & 2BU & 0.092 & 0.609 & 0.774 & 0.590 & 56\% \\
\midrule
10 & 1.05 & 2BU & 0.557 & 1.000 & 0.504 & 0.633 & 0\% \\
10 & 1.2 & 2BU & 0.618 & 0.984 & 0.550 & 0.504 & 0\% \\
10 & 1.5 & 2BU & 0.553 & 0.969 & 1.067 & 1.052 & 5\% \\
10 & 1.8 & He  & 0.064 & 0.375 & 0.061 & 0.003 & 69\% \\
\midrule
12 & 1.05 & 2BU & 0.642 & 1.000 & 0.500 & 0.628 & 0\% \\
12 & 1.2 & 2BU & 0.618 & 1.000 & 0.602 & 0.329 & 0\% \\
12 & 1.5 & 2BU & 0.578 & 1.000 & 0.530 & 0.981 & 0\% \\
12 & 1.8 & He  & 0.063 & 0.344 & 0.096 & 0.008 & 69\% \\
\midrule
20 & 1.05 & 2BU & 0.595 & 1.000 & 0.602 & 0.753 & 0\% \\
20 & 1.2 & 2BU & 0.615 & 1.000 & 0.450 & 0.698 & 0\% \\
20 & 1.5 & 2BU & 0.301 & 0.953 & 0.121 & 0.377 & 12\% \\
20 & 1.8 & 2BU & 0.097 & 0.797 & 0.059 & 0.004 & 56\% \\
\bottomrule
\end{tabular}
\end{table}

To complement the residual-based ranking with a conditional statistical screen, we also evaluate the e-values of \eqref{eq:evalue_screen} where the numerical table given in \Cref{app:evalues}. On the tested subcritical grid, the two-block-uniform family decisively rejects the Hermite baseline and the e-value analysis supports the claim that the $p<2$ branch is non-Hermite, while the descriptor-based AHT comparison indicates that the two-block template is only a coarse surrogate for the underlying structure.

\begin{figure}[htb!]
\centering
\includegraphics[width=\textwidth]{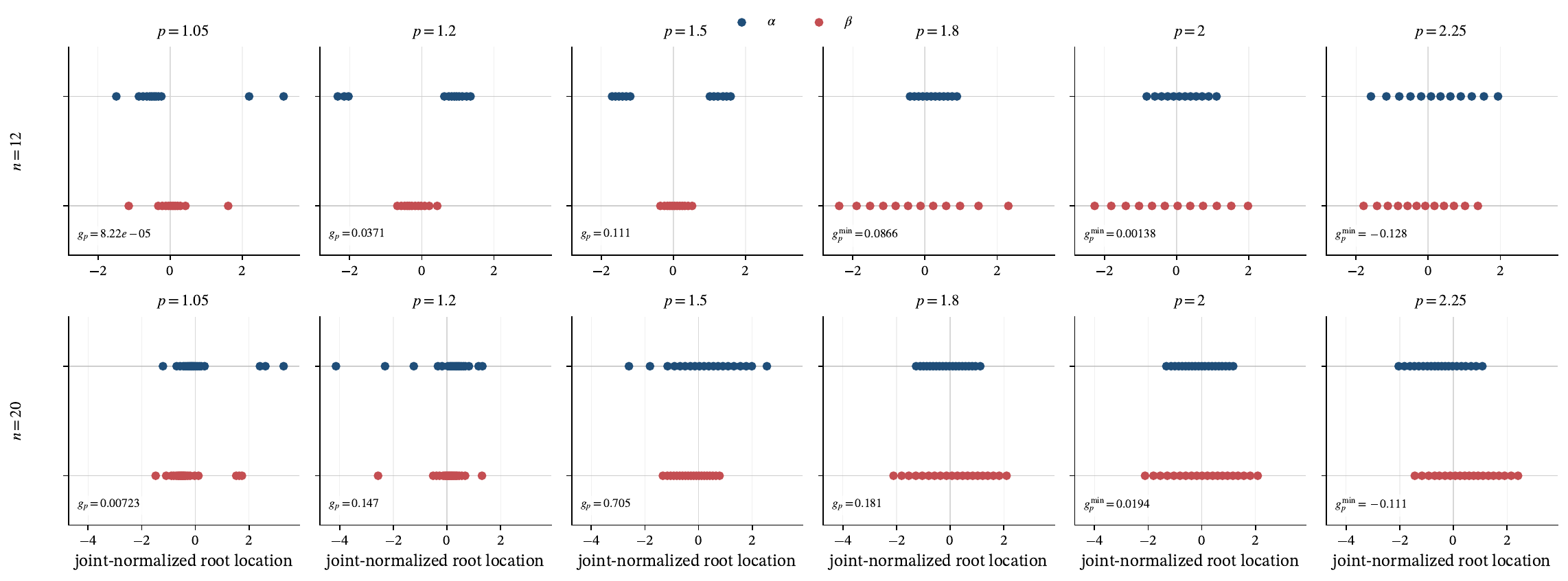}
\caption{Root-configuration atlas for the best observed
configurations at two representative degrees,
$n=12$ and $n=20$, across the transition in~$p$. In each
panel, the blue and red rows show the jointly normalized
root locations of $\alpha$ and $\beta$, respectively,
while the annotated value gives the best observed
additive deficit $g_p$ at that parameter pair. Read from left to right,
the atlas exhibits a clear geometric progression:
non-matching clustered configurations for $p<2$, a
transitional Hermite regime around $p=1.8$, the
Hermite diagonal at $p=2$, and Hermite matching
counterexamples for $p>2$.}
\label{fig:root_atlas}
\end{figure}

\Cref{fig:root_atlas} is the most direct visual evidence
for extremizer bifurcation. It shows that the subcritical extremizers
are not merely noisy perturbations of the Hermite pair,
but qualitatively different configurations whose geometry
deforms continuously and then collapses back to the
Hermite diagonal as $p \to 2$.

\paragraph{Regime I: Deeply bimodal, non-matching
  ($p \in \{1.05,1.2,1.5\}$).}
At $p \in \{1.05,1.2,1.5\}$, the extremizers exhibit
three consistent features across all tested $n$:
(i)~\emph{bimodal root structure} where the two-block
uniform family is the best fit in at least $75\%$ of
retained samples (and in over $95\%$ at all degrees except
$n=8, p=1.05$), with root vectors separating into two
well-separated clusters, 
(ii)~\emph{non-matching pairs} where the mismatch $d_{PQ}$
ranges from $0.12$ to $1.07$, indicating that $f$ and $g$
are structurally distinct, and
(iii)~\emph{absence of diagonal symmetry} where the S3
fraction ($\alpha \approx \beta$) is $0\%$ at tolerance
$t = 0.1$ across all tested pairs.

\paragraph{Regime II: Transitional ($p \approx 1.8$).}
At $p = 1.8$, the pattern is degree-dependent. For
$n \in \{6, 10, 12\}$, the Hermite template is already
the best-fitting family, with $d_{\mathrm{joint}}$
between $0.034$ and $0.064$.
At $n \in \{8, 20\}$, the residual is larger ($d_{\mathrm{joint}} \approx 0.09$--$0.10$), retaining a subcritical signature. This degree-dependent variation marks $p = 1.8$ as a transition band rather than a sharp universal switch: the geometry has moved substantially toward Hermite at most degrees, but the approach is not uniform in $n$.

\paragraph{Regime III: Hermite ($p = 2$).}
At $p = 2$, the extremizer collapses to the unique
matching Hermite pair, with $D \to 0$ and $d_{PQ} \to 0$
(\Cref{sec:hermite}).

The Hermite median $d_{\mathrm{joint}}$ (distance from the best set to the Hermite family) provides a
single-number summary of the bifurcation.
\Cref{tab:hermite_distance_evolution} shows that these
Hermite residuals remain large throughout the strongly
subcritical regime and then collapse sharply by $p=1.8$,
confirming a continuous but highly nonuniform deformation
toward the quadratic endpoint.

\begin{table}[htb!]
\centering
\caption{Hermite median $d_{\mathrm{joint}}$ (distance of best set to Hermite) as a function of $p$.}
\label{tab:hermite_distance_evolution}
\begin{tabular}{@{}r c c c c@{}}
\toprule
$n$ & $p = 1.05$ & $p = 1.2$ & $p = 1.5$ & $p = 1.8$ \\
\midrule
6  & 0.506 & 0.516 & 0.505 & 0.034 \\
8  & 0.712 & 0.643 & 0.548 & 0.112 \\
10 & 0.723 & 0.776 & 0.740 & 0.064 \\
12 & 0.796 & 0.773 & 0.740 & 0.063 \\
20 & 0.739 & 0.727 & 0.394 & 0.111 \\
\bottomrule
\end{tabular}
\end{table}

The $p=1.05$ case is very interesting, showing that even when the objective value is nearly zero at low degree, the Hermite median residual remains of order $0.5$--$0.8$. Thus the equality at small $p$ does not force Hermite geometry. What changes near $p=2$ is not only the size of the deficit, but the shape class of the extremal configurations.

\begin{figure}[htb!]
\centering
\includegraphics[width=\textwidth]{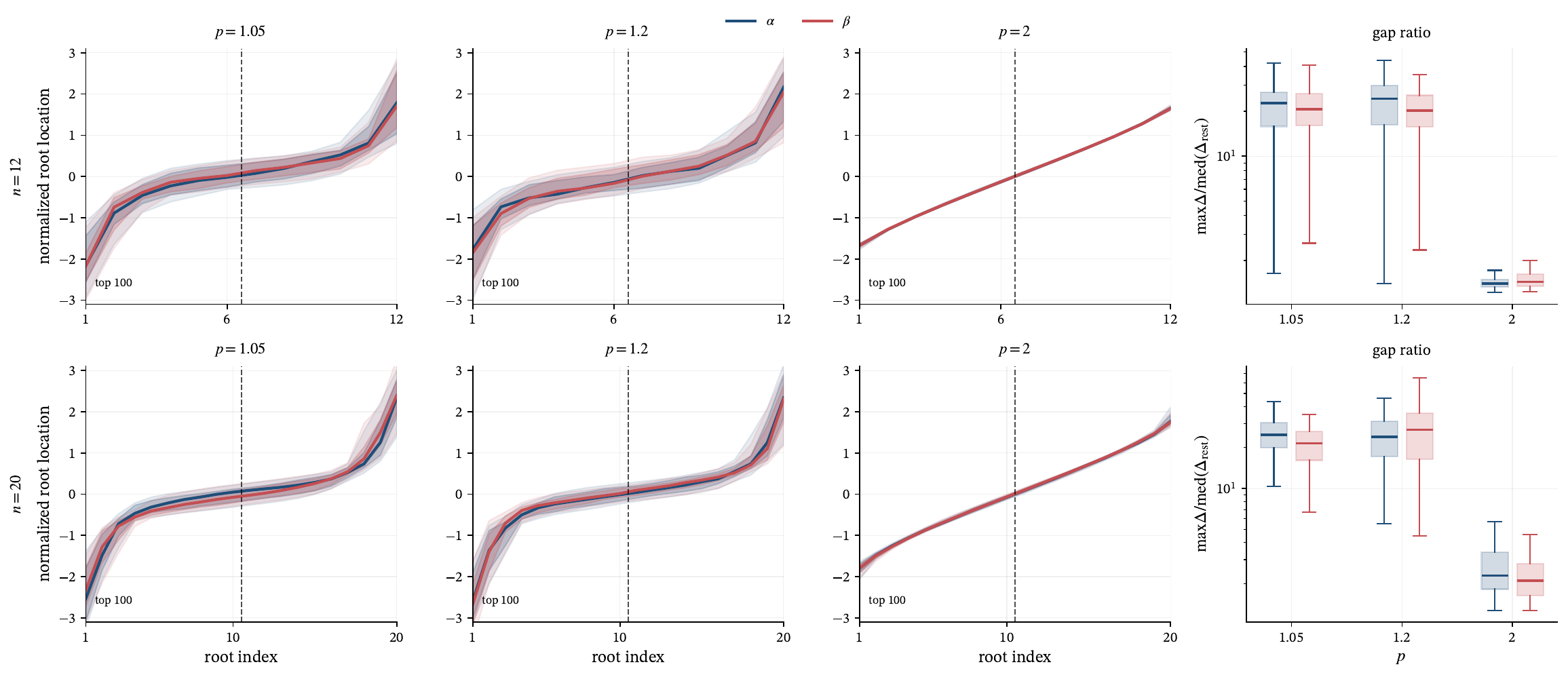}
\caption{Population-level order-statistic bands and gap
analysis for the top-$100$ elite samples at each displayed $(n,p)$ pair. In the first three columns, each polynomial is normalized independently to
mean $0$ and unit variance, and then the median together
with the $25$--$75\%$ and $10$--$90\%$ bands of the
$i$-th sorted root are plotted separately for
$\alpha$ and $\beta$. A two-block population appears here
as a stable jump between root indices $n/2$ and
$n/2+1$. The rightmost column reports the elite
distribution of the scale-free gap statistic
$\max \Delta/\operatorname{med}(\Delta_{\mathrm{rest}})$,
where $\Delta$ denotes adjacent normalized root gaps. Both
diagnostics separate the subcritical cases
$p=1.05,1.2$ from the quadratic endpoint $p=2.0$.}
\label{fig:population_order_statistics}
\end{figure}

\Cref{fig:population_order_statistics} complements
\Cref{fig:root_atlas} by
showing the subcritical deformation directly at the level
of order statistics, rather than through a pooled marginal
density. This is important because it is a population statistic. The figure shows that the structural transition persists after aggregation over samples, which makes the bifurcation phenomena more robust as a collective effect.

\subsection{The imperfection of the two-block fit}

\label{sec:two_cluster}

Although two-block uniform is the best-fitting tested family for $p \in \{1.05, 1.2, 1.5, 1.8\}$, the fit quality is not perfect and the median $d_{\mathrm{joint}}$ ranges from $0.18$ to $0.62$. This indicates that the true extremal configurations have a more refined structure than simple uniform spacing within each block. At present, however, we do not have a model that is simultaneously supported by
the e-value-base AHT and by the \(p\)-Stam objective itself across the tested grid. We
therefore record the two-block pattern only as a coarse description and leave the precise subcritical extremal configuration family open.

\begin{conjecture}[$p$-Stam extremizer bifurcation]
\label{conj:extremizer_transition}
  For $n \ge 3$, the extremizers of the $p$-Stam inequality~\eqref{eq:pstam} undergo a bifurcation as $p$ decreases from $2$:
  \begin{enumerate}[nosep, label=(\roman*)]
    \item \textbf{Non-matching:} for $p \in (1, 2)$,
      the normalized extremizers satisfy $f \not\approx g$, with mismatch
      $d_{PQ} > 0$ that increases as $p \to 1^+$.
    \item \textbf{Bimodal:} each polynomial's root
      distribution is bimodal, with roots clustering into
      two groups separated by a gap. The bimodality intensifies as $p \to 1^+$.
    \item \textbf{Continuous transition:} as $p \to 2^-$,
      the bimodal structure merges, the mismatch vanishes,
      and the extremizers converge continuously to the
      matching Hermite pair
      $(\mathrm{He}_n, \mathrm{He}_n)$.
  \end{enumerate}
\end{conjecture}

\begin{figure}[htb!]
\centering
\includegraphics[width=\textwidth]{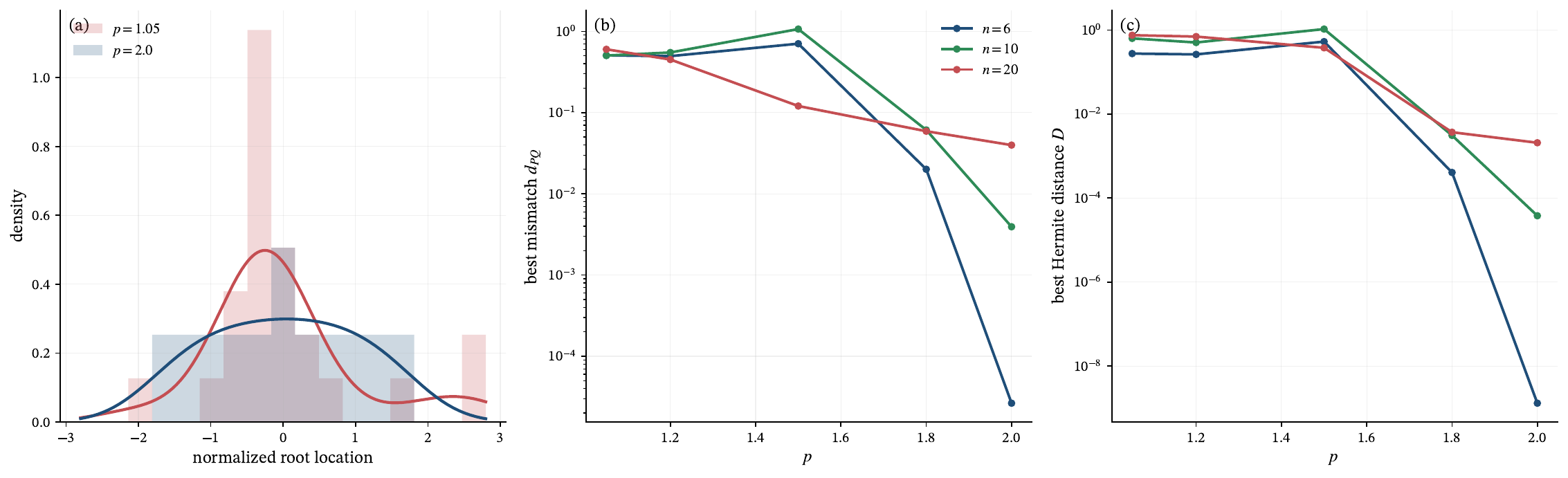}
\caption{Compact summary of the subcritical bifurcation.
Panel~(a) contrasts a representative low-$p$ root
distribution with the quadratic Hermite geometry;
Panel~(b) tracks the best pair-mismatch statistic
$d_{PQ}$ across~$p$; and Panel~(c) tracks the best Hermite
distance~$D$. Together the panels summarize the two main
structural signatures of the transition: mismatch and
non-Hermite geometry both persist well below $p=2$ and
collapse rapidly as $p \to 2^-$.}
\label{fig:extremizer_transition}
\end{figure}

\Cref{fig:extremizer_transition} summarizes the geometric transition seen in \Cref{fig:root_atlas}. The minimum deficits found by FlowBoost on the tested subcritical grid are reported in \Cref{tab:pstam_lt2}. These values remain positive, but some are at the same scale as the numerical recovery error observed at the known equality case $p=2$. Accordingly, the table should be read only as lower-envelope evidence, not as numerical proof of strictness. The more robust conclusion is that for $p<2$, the minimizers found by FlowBoost remain non-Hermite even when the deficit is very small, and return toward the Hermite branch as $p\to 2$.

\begin{conjecture}[$p$-Stam strict inequality]
\label{conj:strict}
  For every $p \in (1,2)$, every $n \ge 3$, and every
  pair of monic real-rooted polynomials $f, g$ of
  degree $n$, $g_p(f,g) > 0.$ In particular, the $p$-Stam inequality admits no
  equality case for $p \in (1,2)$.
\end{conjecture}

\begin{remark}
\label{rem:strict_vs_extremizers}
The strictness conjecture is consistent with the non-Hermite extremizer
geometry observed in \Cref{sec:extremizers_p}. In particular, the best
$p<2$ candidates found by FlowBoost are typically non-diagonal
($f\neq g$ after normalization) and exhibit clustered or bimodal
root structure, rather than collapsing onto the Hermite equality regime
seen at $p=2$. Any equality case for $p<2$, if it exists, must lie strictly off the Hermite diagonal.
\end{remark}

\subsection{The $p \to 1^+$ regime}
\label{sec:p_to_1}
The limit $p \to 1^+$ is numerically singular. The exponent $1/(p-1)$ diverges, so small multiplicative changes in the normalized Fisher-information terms are strongly amplified in the deficit. Our computations do not identify a closed-form limiting functional, and we do not claim a rigorous minimax description of this regime. What the data do show is more modest and more robust. At $p=1.05$, the best observed deficits can already be very small at low degree, yet the associated best elite populations remain decisively non-Hermite and are still best fit by the two-block template across the tested degrees. Thus the strongly subcritical regime is not a small perturbation of the Hermite equality geometry. Any future analysis of the $p \to 1^+$ limit will therefore have to explain how the equality can coexist with a global two-cluster structure.

\section{Conclusion}\label{sec:discussion}
This work demonstrates that closed-loop generative optimization, can be an effective tool for mathematical discovery. The four contributions of this paper span two qualitatively different modes of mathematical discovery. The spectral conjecture and the Hermite deficit formula were identified through numerical pattern finding once FlowBoost pointed to the relevant configuration. The extremal configuration bifurcation for $p < 2$ emerged from global adversarial search with no prior indication of where to look. The key paradigm is \emph{adversarial probing}: rather than asking ``does this inequality hold?'', FlowBoost asks ``where does this inequality most nearly fail?'' and ``what do the extremal configurations look like?'' These are global optimization questions whose answers reveal structure, phase transitions, extremizer families, and spectral gaps. Our goal has been to show that FlowBoost can serve both as a guide for analytic work, by pointing to the right configurations, and as a discovery framework for new phenomena.

\subsection{Future Directions}\label{sec:future}
The results of this paper raise several natural problems for future work. \Cref{conj:spectrum} posits that $\spec(E_n|_W)=\{2^{-k/2}:k=1,\ldots,n-1\}$ independently of $n$. A proof would require understanding the spectral theory of the linearized finite free convolution map at the Hermite diagonal, and in particular why the dyadic pattern persists across all polynomial degrees. Moreover, \Cref{conj:pstam_holds} asserts that the finite free Stam inequality holds for all $\ell^p$-Fisher information with $p \in (1,2]$. The GVS26 proof at $p=2$ uses a quadratic Hessian positivity that does not extend directly to $p \ne 2$. A proof for $p < 2$ would require a new positivity mechanism. The e-value-based AHT screen robustly supports a two-block structure for $p \lesssim 1.5$ but the two-block-uniform template has high residuals and are too simple. The precise closed-form description of the $p < 2$ extremizer family is open. Lastly, \Cref{rem:uniqueness} records the numerical evidence that the Hermite diagonal is the unique equality case of \Cref{thm:stam}. The GVS26 proof does not address uniqueness. A proof would follow from a strict-convexity refinement of the Hessian positivity argument away from the Hermite orbit.

\appendix
\section{Further Computational Details}

\subsection{Numerical computation for \Cref{conj:spectrum}}
\label{app:spectrum_audit}

This appendix records the computation underlying the
numerical evidence for \Cref{conj:spectrum}. The
conjecture concerns the action of $E_n$ on
$W=\one^\perp$, so we audit the induced operator on $W$
directly rather than reconstructing the full ambient
Jacobian numerically. At the Hermite diagonal the
convolution output is the rescaled Hermite configuration
$\Omega_{\boxplus_n}(h^{(n)},h^{(n)})=\sqrt2\,h^{(n)}$~\cite{Marcus21}.
If $P(x)$ denotes the convolution polynomial with roots
$\gamma_i=\sqrt2\,h_i^{(n)}$, then for a perturbation of a
single input root the corresponding coefficient
perturbation $\delta P$ is explicit, and implicit
differentiation of $P(\gamma_i)=0$ gives $\delta \gamma_i \;=\; -\frac{\delta P(\gamma_i)}{P'(\gamma_i)}.$
Applying this formula coordinatewise yields a
representative Jacobian block at the Hermite point. Its
ambient matrix is defined only up to a rank-one
translation-gauge term, but the induced operator on
$W=\one^\perp$ is canonical, which is exactly the object
relevant to \Cref{conj:spectrum}. We compute this restricted operator with multiprecision
coefficient arithmetic using \texttt{mpmath}~\cite{mpmath}.
The audited degrees are
$n\in\{10,20,30,40,50,60,70,80,90\}$, with working
precision increased from $60$ digits at $n=10$ to
$220$ digits at $n=90$. For each degree we compute the
first ten singular values of $E_n|_W$ and compare them to
the dyadic target $2^{-k/2}$. The resulting maximum
relative errors are
$1.33\times10^{-14},\,5.48\times10^{-13},\,8.69\times10^{-12},\,8.02\times10^{-11},\,2.13\times10^{-9},\,1.15\times10^{-7},\,1.05\times10^{-6},\,5.77\times10^{-5},\,4.56\times10^{-4}$
for
$n=10,20,30,40,50,60,70,80,90$, respectively. These are
the values summarized in \Cref{tab:spectrum} and
\Cref{fig:spectrum}.

\subsection{E-value based Abductive Hypothesis Testing}
\label{app:evalues}

This appendix records the supplementary e-value analysis
used in \Cref{sec:extremizers_p}. The e-values are defined
in \eqref{eq:evalue_screen} under the working model that,
for fixed $(n,p)$, the retained best elite configuration is
exchangeable and generated from a fixed law. The comparison is performed samplewise using the Mode~A joint residuals. For the Hermite-null screen, the indicator is
$Z_m^{\mathcal H}=\mathbf{1}\{R_{\mathcal H}(x^{(m)})<
R_{\mathrm{He}}(x^{(m)})\}$ after removing ties. The
reported e-values should therefore be read as conditional
model-comparison diagnostics in the sense of Vovk and
Wang~\cite{VovkWang21}, not as certificates for the
adaptive FlowBoost search itself.

\begin{table}[htb!]
\centering
\caption{Conditional e-values for the one-sided screen
$H_{0}^{\mathrm{He}}(\mathrm{2BU};n,p):
\mathbb{P}(R_{\mathrm{2BU}}<R_{\mathrm{He}})\le 1/2$
against the alternative that two-block uniform beats
Hermite more than half the time on the top-$512$ elite.
$M_{\mathrm{eff}}$ is the number of non-tied comparisons
after removing samples for which both residuals are equal.
Cells where the e-value saturates at $5.28\times10^{36}$ correspond
to all non-tied comparisons won by two-block uniform.
The decision column uses the $5\%$ e-value rule $E\ge 20$.}
\label{tab:evalue_tb_vs_he}
\scriptsize
\setlength{\tabcolsep}{8pt}
\begin{tabular}{@{}r c c c c@{}}
\toprule
$n$ & $p$ & $E$ & Reject $H_0$ & Favoured model \\
\midrule
6  & 1.05 & $5.28\times 10^{36}$ & Yes & Two-block uniform \\
6  & 1.2  & $5.28\times 10^{36}$ & Yes & Two-block uniform \\
6  & 1.5  & $5.28\times 10^{36}$ & Yes & Two-block uniform \\
6  & 1.8  & 6.831              & No  & Inconclusive \\
6  & 2.0  & 0.008               & No  & Hermite \\
\midrule
8  & 1.05 & $4.95\times 10^{29}$ & Yes & Two-block uniform \\
8  & 1.2  & $9.73\times 10^{26}$ & Yes & Two-block uniform \\
8  & 1.5  & $9.73\times 10^{26}$ & Yes & Two-block uniform \\
8  & 1.8  & 35.81                & Yes & Two-block uniform \\
8  & 2.0  & 0.008               & No  & Hermite \\
\midrule
10 & 1.05 & $6.49\times 10^{32}$ & Yes & Two-block uniform \\
10 & 1.2  & $6.49\times 10^{32}$ & Yes & Two-block uniform \\
10 & 1.5  & $4.95\times 10^{29}$ & Yes & Two-block uniform \\
10 & 1.8  & 5.679              & No  & Inconclusive \\
10 & 2.0  & 0.008               & No  & Hermite \\
\midrule
12 & 1.05 & $5.28\times 10^{36}$ & Yes & Two-block uniform \\
12 & 1.2  & $5.28\times 10^{36}$ & Yes & Two-block uniform \\
12 & 1.5  & $5.28\times 10^{36}$ & Yes & Two-block uniform \\
12  & 1.8  & 23.65                & Yes & Two-block uniform \\
12 & 2.0  & 0.008              & No  & Hermite \\
\midrule
20 & 1.05 & $5.28\times 10^{36}$ & Yes & Two-block uniform \\
20 & 1.2  & $5.28\times 10^{36}$ & Yes & Two-block uniform \\
20 & 1.5  & $9.73\times 10^{26}$ & Yes & Two-block uniform \\
20 & 1.8  & 10.18              & No  & Inconclusive \\
20 & 2.0  & 0.008              & No  & Hermite \\
\bottomrule
\end{tabular}
\end{table}

The table shows the intended use of the screen. Even though two-block uniform is only a coarse surrogate, it already rejects the Hermite explanation throughout the tested subcritical grid, with the weakest evidence at the lowest degrees near the transition band. We use this evidence to reinforce the rejection of the Hermite branch, while leaving the precise refined subcritical family open.

\subsection{FlowBoost Stages: A diagnosis}
\Cref{fig:optimization_trace} is a methodological diagnostic result. Here, we want to explain how FlowBoost should be interpreted as a closed-loop staged search procedure in which derivative-free extremal discovery, learned proposal distributions, and local refinement each serve distinct roles in driving the search toward the best observed configurations.

\begin{figure}[htb!]
\centering
\includegraphics[width=0.82\textwidth]{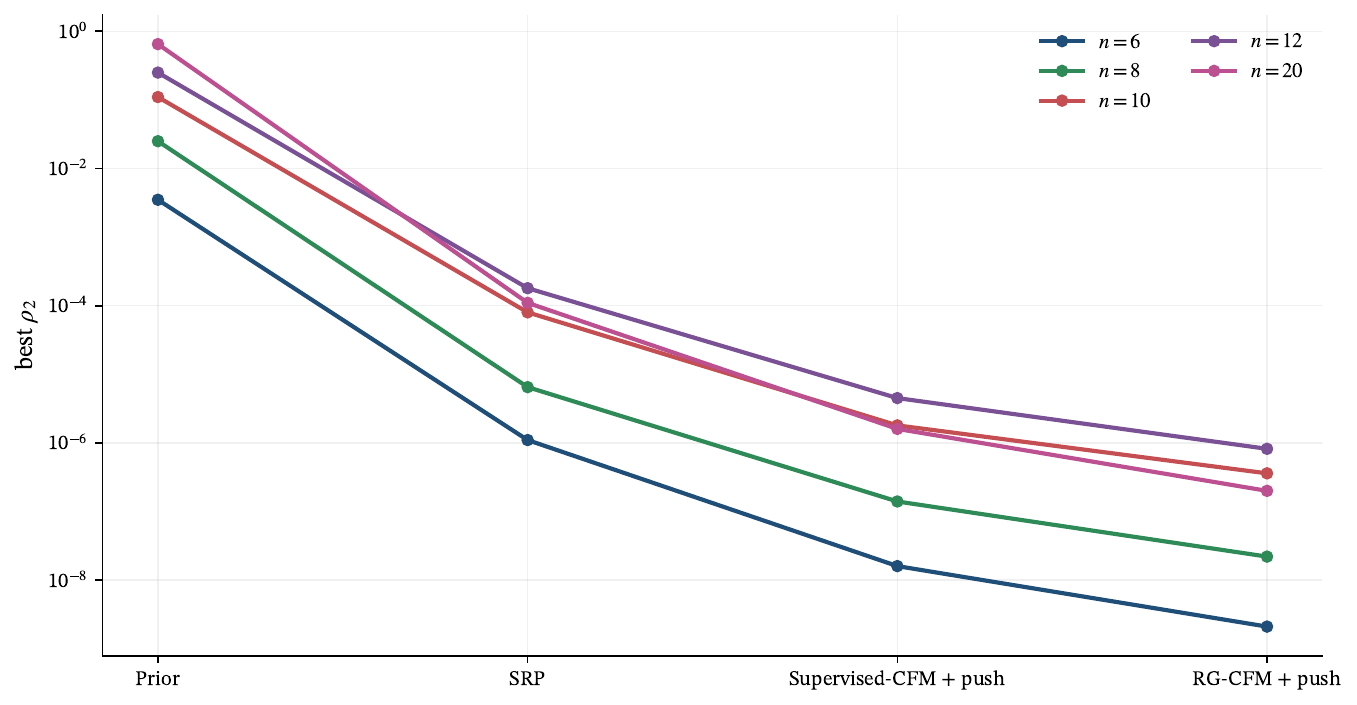}
\caption{Stage-wise results for the $p=2$ FlowBoost pipeline at representative degrees $n=6,8,10,12,20$. SRP first identifies a high-quality feasible region,
supervised CFM learns a proposal distribution from that region, and reward-guided fine-tuning sharpens the proposal distribution further so that the subsequent final push attains the best values.}
\label{fig:optimization_trace}
\end{figure}

\FloatBarrier

\bibliographystyle{unsrt}
\bibliography{references}

\end{document}